\documentclass[11pt,reqno]{amsart}

\usepackage{enumerate}
\usepackage{amsmath, amssymb, amsthm}
\usepackage{mathrsfs}
\usepackage{esint}
\usepackage{xcolor}
\usepackage{mathtools}
\usepackage{hyperref}
\usepackage{bm}
\usepackage{kotex}

\newtheorem{thm}{Theorem}[section]
\newtheorem{corollary}[thm]{Corollary}
\newtheorem{lem}[thm]{Lemma}

\theoremstyle{definition}
\newtheorem{defn}[thm]{Definition}
\newtheorem{example}[thm]{Example}
\newtheorem{assumption}[thm]{Assumption}
\theoremstyle{remark}
\newtheorem{rem}[thm]{Remark}
\newtheorem*{acknowledgment}{Acknowledgments}

\newcommand\bL{\mathbb{L}}
\newcommand\bR{\mathbb{R}}
\newcommand\bC{\mathbb{C}}
\newcommand\bH{\mathbb{H}}

\newcommand\bZ{\mathbb{Z}}

\newcommand\bE{\mathbb{E}}
\newcommand\bN{\mathbb{N}}
\newcommand\bP{\mathbb{P}}

\newcommand\cB{\mathcal{B}}
\newcommand\cC{\mathcal{C}}

\newcommand\cF{\mathcal{F}}

\newcommand\cH{\mathcal{H}}

\newcommand\cP{\mathcal{P}}

\newcommand\cS{\mathcal{S}}
\newcommand\cM{\mathcal{M}}

\newcommand\cbrk{\text{$]$\kern-.15em$]$}}
\newcommand\opar{\text{\,\raise.2ex\hbox{${\scriptstyle|}$}\kern-.34em$($}}
\newcommand\cpar{\text{$)$\kern-.34em\raise.2ex\hbox{${\scriptstyle |}$}}\,}
\newcommand\ep{\varepsilon}

\newcommand{\mysection}[1]{\section{#1}
\setcounter{equation}{0}}

\begin{document}

\title[SPDEs with a super-linear diffusion coefficient and a colored noise]{A regularity theory for stochastic partial differential equations with a super-linear diffusion coefficient and a spatially homogeneous colored noise }

\author[J.-H. Choi]{Jae-Hwan Choi}
\address[J.-H. Choi]{Department of mathematics, Korea University, 145 Anamro, Seongbukgu, Seoul, Seoul, Republic of Korea}
\email{choijh1223@korea.ac.kr}

\author[B.-S. Han]{Beom-Seok Han}
\address[B.-S. Han]{Department of mathematics, Korea University, 145 Anamro, Seongbukgu, Seoul, Seoul, Republic of Korea}
\email{hanbeom@korea.ac.kr}
\thanks{The authors have been supported by the National Research Foundation of Korea(NRF) grant funded by the Korea government(MSIT) (No. NRF-2019R1A5A1028324)}



\maketitle

\begin{abstract}

Existence, uniqueness, and regularity of a strong solution are obtained for stochastic PDEs with a colored noise $F$ and its super-linear diffusion coefficient: 
$$ 
du=(a^{ij}u_{x^ix^j}+b^iu_{x^i}+cu)dt+\xi|u|^{1+\lambda}dF, 
\quad (t,x)\in(0,\infty)\times\mathbb{R}^d,
$$
where $\lambda \geq 0$ and the coefficients depend on $(\omega,t,x)$. 
The strategy of handling nonlinearity of the diffusion coefficient is to find a sharp estimation for a general Lipschitz case,
and apply it to the super-linear case. 
Moreover, investigation for the estimate provides a range of $\lambda$, a sufficient condition for the unique solvability, where the range depends on the spatial covariance of $F$ and the spatial dimension $d$. 
\end{abstract}

\mysection{Introduction}

In the research on nonlinear stochastic partial differential equations (SPDEs), problems with nonlinear diffusion coefficients and their solution behaviors have been studied \cite{burdzy2010non, Kijung,mueller1999critical,mueller2009some,mueller2014nonuniqueness,
mytnik2011pathwise,skorokhod1982studies,Xiong,yamada1971uniqueness}  and applied to various other fields, such as chemistry \cite{arnold1981mathematical}, neurophysiology \cite{henry2006geometric,mckean1970nagumo,walsh1986introduction}, and population dynamics \cite{dawson1972stochastic}. 
In this article, we investigate the existence, uniqueness, and regularity of a solution to 
\begin{equation}\label{origineq}
\begin{aligned}
du=(a^{ij}u_{x^ix^j}+b^iu_{x^i}+cu)dt+\xi|u|^{1+\lambda} dF,  \,\, (t,x)\in(0,\infty)\times\bR^d; \,\, u(0,\cdot)=u_0, \\
\end{aligned}
\end{equation}
where $\lambda\geq0$. The coefficients $a^{ij},b^{i}$, and $c$ are $\cP\times\cB(\bR^d)$-measurable and twice continuously differentiable in $x$. The coefficient $\xi$ is $\cP\times\cB(\bR^d)$-measurable and bounded. The initial data $u_0(x)$ is random and nonnegative. The noise $F(t,x)$, referred to as the \textit{spatially homogeneous colored noise} \cite{dawson1980spatially}, is a centered Gaussian noise that is white in time and homogeneously colored in space, with covariance given by
$${\bE[F(t,x)F(s,y)]}=\delta_0(t-s)f(x-y),$$
where $\delta_0$ is the centered Dirac delta distribution. The distribution $f$ is {\it nonnegative} and {\it nonnegative definite}; see Definition \ref{191024_1}.

Next, we provide some motivation and history for this problem. 
An SPDE with a super-linear diffusion coefficient has been studied \cite{han2020boundary,kry99analytic,mueller1991long}:
\begin{equation}
\label{190917_01}
du = (au_{xx}+bu_x+cu)\,dt+\xi|u|^{1+\lambda}dW,\quad t>0;\quad u(0,\cdot) = u_0,
\end{equation}
where $\lambda$ is in $ [0,1/2)$, $u_0 \in C_c^\infty$ is nonnegative, and $W$ is a Gaussian noise that is white both in time and space. In \cite{mueller1991long}, C. Mueller proves the long-time existence of  solutions for the particular case where $a = 1$, $b = c = 0$, and $\xi = 1$. In \cite[Section 8.4]{kry99analytic}, N. Krylov proves the existence and  regularity of solutions to equation \eqref{190917_01}  with random and space and time-dependent coefficients. In \cite{han2020boundary}, the unique solvability in weighted Sobolev spaces is obtained and boundary behavior of the solution on a finite interval is studied.

Yet, it turns out that if the spatial dimension $d \geq 2$, an SPDE \eqref{190917_01} with the space-time white noise cannot have a real-valued solution; see e.g. \cite{walsh1986introduction}. 
To overcome this weakness, for higher-dimensional cases, the spatially  homogeneous colored noise has been studied \cite{dalang1998stochastic,dalang1999extending,dawson1980spatially,ferrante2006spdes}, as an alternation of the space-time white noise. 
In \cite{dawson1980spatially}, Dawson introduces the notion of the spatially homogeneous colored noise $F$ and proves the existence and uniqueness of a solution to 
\begin{equation}
\label{190917_02}
du = \Delta u dt + \sigma(u)dF;\quad u(0,\cdot) = u_0
\end{equation} 
with nonnegative constant initial data $u_0 = c\geq0$ and a linear diffusion coefficient $\sigma(u) = u$. In \cite{dalang1998stochastic,dalang1999extending}, $F(dt,dx)$ is interpreted as a martingale measure by employing Walsh's approach \cite{walsh1986introduction}. Here, the existence and uniqueness of a solution are obtained with $u_0 = c$ and $\sigma(u)$ satisfying the Lipschitz condition. In \cite{ferrante2006spdes}, inspired by \cite{kry99analytic},  the $L_p$-solvability is proved under the following conditions on $\sigma(u)$: for any $x\in \bR^d$ and $u\in \bR$, $\sigma(\omega,t,x,u)$ is predictable and 
\begin{equation*}
|\sigma(\omega,t,x,u) - \sigma(\omega,t,x,v)| \leq N|u-v|,
\end{equation*}
where $N$ is independent of $\omega,t,x,u,$ and $v$.

It should be remarked that, to the best of our knowledge, the studies on SPDEs with the colored noise $F$ have proceeded to equations with the Lipschitz diffusion coefficient only. 
A key objective of this paper is to extend the research scope and study SPDE \eqref{origineq} with a super-linear diffusion coefficient by employing the approach of \cite{kry99analytic}.

There are three notable achievements of this study. 
Firstly, 
the conditions for the initial data and the coefficients assumed in \cite{dalang1998stochastic,dalang1999extending,dawson1980spatially} are generalized: the initial data $u_0$ is random, and the coefficients $a^{ij}, b^i, c$ and $\xi$ depend on $\omega,t$, and $x$.

Secondly, 
we prove the existence and uniqueness of a strong solution and provide maximal H\"older regularity of the solution. The following basic example clarifies the maximal regularity of the strong solution to SPDE \eqref{origineq} in a multi-dimensional space: 
assume that $\lambda\in [0,1/d)$ and the spatial covariance {\it{function}} $f$ is nonnegative and nonnegative definite. 
If $u_0$ is nonnegative and in $L_p(\Omega;H_p^{1/2-\lambda}(\bR^d))\cap L_1(\Omega;L_1(\bR^d))$ for all $p>2$, then almost surely for any $T<\infty$ and small $\ep>0$, 
\begin{equation} \label{eq 19.12.21}
u\in C^{\frac{\gamma}{2}-\varepsilon,\gamma-\varepsilon}_{t,x}\left([0,T]\times \bR^d \right),
\end{equation}
where $\gamma = \frac{1}{2}\wedge (1-\lambda d)$. 
In other words, $u$ satisfies almost surely
$$
\sup_{t\in[0,T]}|u(t,\cdot)|_{C^{\gamma-\varepsilon}(\bR^d)}<\infty\quad\text{and} \quad\sup_{x\in\bR^d}|u(\cdot,x)|_{C^{\frac{\gamma}{2}-\varepsilon}([0,T])}<\infty.
$$

It is worthwhile to note that if we consider a {\it{general}} covariance, the number of regularity changes; see Corollary \ref{190916_01}.

Lastly, 
a range of $\lambda$ and conditions on $f$ are obtained as a sufficient condition for unique solvability of \eqref{origineq} in $L_p$-spaces. To see this, we investigate 
relations between the $L_p$-norm of diffusion coefficients and the $L_p$-norm of a solution, which leads to a sharp estimate on the diffusion coefficients.
Moreover, it is found out that summability of the coefficient $\xi$ has a critical effect on the range of $\lambda$; see Remark \ref{rmk:range_of_lambda}.

This paper is organized as follows. 
Section \ref{sec:preliminaries} introduces preliminary definitions and properties. 
In Section \ref{sec:main}, major results are presented for a general Lipschitz diffusion coefficient and for the super-linear coefficient. 
In Sections \ref{sec:lambda=0} and \ref{sec:superlinear}, proofs for the Lipschitz case and the super-linear case are elaborated, respectively.

We finish this section with the notations used in the article.
$\bN$, $\bZ$, and $\bC$ denote the set of natural numbers, integers, and complex numbers, respectively. $\bR$ is the set of all real numbers and $\bR^d$ stands for the $d$-dimensional Euclidean space of points $x = (x^1,\dots,x^d)$ for $x^i\in \bR$, $i=1,2,\dots,d$. { We use  `$:=$' to denote a definition. Let $M$ and $C_c^{\infty}=C_c^{\infty}(\bR^d)$ denote the set of nonnegative Borel measures on $\bR^d$, and  the set of  real-valued infinitely differentiable functions with compact support on $\bR^d$, respectively. $\cS=\cS(\bR^d)$ stands for the space of  Schwartz functions on $\bR^d$. $\cS_R$ denotes a set of real-valued Schwartz functions on $\bR^d$. Let $\cS'=\cS'(\bR^d)$ be the space of tempered distributions on $\bR^d$.} For $f,g\in \cS$, let us denote by
\begin{equation*}
\cF(f)(\xi):=\frac{1}{(2\pi)^{d/2}}\int_{\bR^d} e^{-i\xi\cdot x}f(x)dx,\quad \cF^{-1}(g)(x):=\frac{1}{(2\pi)^{d/2}}\int_{\bR^d} e^{i\xi\cdot x}g(\xi)d\xi,
\end{equation*}
the Fourier transform of $f$ in $\bR^d$ and inverse Fourier transform of $g$, respectively. 
For $p \in [1,\infty)$, a normed space $F$,
and a  measure space $(X,\mathcal{M},\mu)$, a function space $L_{p}(X,\cM,\mu;F)$
denotes the space of all $F$-valued $\mathcal{M}^{\mu}$-measurable functions
$u$ such that
\[
\left\Vert u\right\Vert _{L_{p}(X,\cM,\mu;F)}:=\left(\int_{X}\left\Vert u(x)\right\Vert _{F}^{p}\mu(dx)\right)^{1/p}<\infty,
\]
where $\mathcal{M}^{\mu}$ is the completion of $\cM$ with respect to the measure $\mu$.
For $a,b\in \bR$, 
$a \wedge b := \min\{a,b\}$, $a \vee b := \max\{a,b\}$.  For  $\mathbf{a}\in \bC$, $\bar{\mathbf{a}}$ denotes the complex conjugate of $\mathbf{a}$. For a function $\phi$ on $\bR^d$, we define $\tilde{\phi}(x) := \phi(-x)$.
A generic constant is denoted by $N$.
Writing $N = N(a,b,\cdots)$ implies that the constant $N$ depends only on $a,b,\cdots$. Finally,
for functions  depending on $\omega$, $t$,  and $x$, the argument
$\omega \in \Omega$ is omitted.

\mysection{Preliminaries} 
\label{sec:preliminaries}

In this section, basic definitions and related notions are introduced
on distributions, spatially homogeneous colored noises, and stochastic Banach spaces.

\subsection{Definitions and properties related to distributions}

\begin{defn}\label{191024_1}
\begin{enumerate}[(i)]
\item A $\bC$-valued continuous function $f$ is \textit{nonnegative definite} if 
\begin{equation*}
\int_{\bR^d}\int_{\bR^d}f(x-y)\psi(x)\bar{\psi}(y)dxdy=(f,\psi*\tilde{\bar{\psi}})\geq0 \quad \forall \psi\in \cS.
\end{equation*}

\item A distribution $f\in\cS'$ is \textit{nonnegative} if
\begin{equation*}
\label{positive}
(f,\psi)\geq 0 \quad \text{for all  nonnegative}~ \psi\in\cS_R.
\end{equation*}

\item  A distribution $f\in \cS'$ is \textit{nonnegative definite} if 
$$(f,\psi\ast\bar{\tilde{\psi}})\geq0\quad \forall \psi\in \cS.
$$

\item A  nonnegative Borel measure  $\mu \in M$ is  a \textit{tempered measure} if there exists $k\in[0,\infty)$ such that
\begin{equation}
\label{191024_02}
\int_{\bR^d}\frac{1}{(1+|x|^2)^{k/2}}\mu(dx)<\infty.
\end{equation}
\end{enumerate}
Let $C_{nd}$, $\cS'_{n}$, $\cS'_{nd}$, and $M_T$ denote the sets containing the objects defined in (i), (ii), (iii) and (iv), respectively.
\end{defn}

Below we collect some well-known results on the objects defined in Definition \ref{191024_1}; see \cite{gel2014generalized}.

\begin{thm} 
\label{191024_03}
\begin{enumerate}[(i)]
\item 
Let $M_F$ be the set of finite Borel measures and $L:M_F\to C_{nd}$ be a mapping such that
$$ (L\mu)(x):=\int_{\bR^d}e^{i\xi\cdot x}\mu(d\xi).$$
Then the mapping $L$ is bijective. 
\item 
There exists an one-to-one correspondence between $\cS'_{n}$ and $M_T$. Precisely, for any $f\in\cS'_{n}$, there exists a unique $\mu\in M_T$ such that
\begin{equation}
\label{191025_01}
(f,\psi)=\int_{\bR^d}\psi(x)\mu(dx)\quad \forall \psi\in\cS.
\end{equation}
Conversely, if $\mu\in M_T$ is given, there exists a unique $f\in \cS'_n$ satisfying \eqref{191025_01}.
For a given $f\in \cS_n'$, the corresponding $\mu = \mu_f$ is referred to as the corresponding measure of $f$.

\item There exists an one-to-one correspondence between $\cS'_{nd}$ and $M_T$. Precisely, for any $f\in\cS'_{nd}$, there exists a unique $\nu\in M_T$ such that
\begin{equation}
\label{191107_01}
(f,\psi)=\int_{\bR^d}\overline{\cF(\psi)}(\xi)\nu(d\xi)\quad \forall \psi\in\cS.
\end{equation}
Conversely, if $\nu\in M_T$ is given, there exists a unique $f\in \cS'_{nd}$ satisfying \eqref{191107_01}. 
For a given $f\in \cS_{nd}$, the corresponding $\nu = \nu_{\cF(f)}$ is referred to as the corresponding measure of $\cF(f)$.

\item For $f\in\cS'_{nd}$, there exist  $h\in C_{nd}$ and $l\in[0,\infty)$ such that
\begin{equation*}
f=(1-\Delta)^{l/2}h\quad\mbox{and}\quad h(x)=\int_{\bR^d}\frac{e^{i\xi\cdot x}}{(1+|\xi|^2)^{l/2}}\nu(d\xi),
\end{equation*}
where $\nu$ is reffered to as the corresponding measure of $\cF(f)$.
\end{enumerate}
\end{thm}
\begin{proof}
See \cite{gel2014generalized}.
\end{proof}

\begin{rem}
\label{191108_01}

\begin{enumerate}[(i)]

\item For any given $\mu\in M_T$, $\cS\subseteq L_p(\mu)$ for all $p\in[1,\infty]$. 

\item If $f\in \cS'_{n}\cap C_{nd}$, then $f$ is a real-valued, nonnegative, symmetric, continuous, and  \textit{bounded} function. Thus, \eqref{191025_01} can be written as
\begin{equation}
\label{poft'}
(f,\phi) = \int_{\bR^d}\phi(x)\mu(dx)=\int_{\bR^d}\phi(x)f(x)dx \quad\forall\phi\in \cS.
\end{equation}
For more details, see \cite{gel2014generalized}.
\end{enumerate}
\end{rem}

\vspace{2mm}

\subsection{Spatially homogeneous colored noise}

Here and thereafter, let $(\Omega, \cF, \bP)$ be a complete probability space with a filtration $\{\cF_t\}_{t\geq0}$ satisfying the usual conditions, $\cP$ be the predictable $\sigma$-field related to $\cF_t$, and $f$ be a distribution such that 
$$f\in \cS'_{n}\cap\cS'_{nd}.
$$

\begin{defn}
The random noise $F(t,A)$ is called {\it a spatially homogeneous colored noise with covariance $f$} if $F(t,A):=F(1_{[0,t]}1_A)$ is a centered Gaussian process and its covariance is given by
\begin{equation} 
\label{cov}
\bE[F(t,A)F(s,B)]=t\wedge s(f,1_{A}*\tilde{1}_{B}),
\end{equation}
where $A$ and $B$ are bounded Borel sets and $t\geq0$.
\end{defn}

\begin{rem} \label{remark:representation}
The framework of Walsh \cite{walsh1986introduction} has been used to define a stochastic integral with respect to $F(dt,dx)$; see \cite{dalang1998stochastic,dalang1999extending}. On the other hand, the integral can also be written as an infinite summation of It\^o stochastic integral; for  predictable process $X(t,\cdot)$ such that
$$ X(t,x) = \zeta(x)1_{ \opar\tau_1,\tau_2\cbrk}(t)
$$
where $\tau_1$, $\tau_2$ are bounded stopping times, $\opar\tau_1,\tau_2\cbrk:=\{ (\omega,t):\tau_1(\omega)<t\leq \tau_2(\omega) \}$, and $\zeta\in C_c^\infty$, we have 
\begin{equation*} 
\int_{0}^t\int_{\bR^d}X(s,x)F(ds,dx)=\sum_{k=1}^{\infty}\int_0^t \int_{\bR^d}(f\ast e_k)(x) X(s,x) dxdw^k_s,
\end{equation*}
where $w_t^k, k \in \bN,$ is a one-dimensional independent Wiener process and $e_k(x), k \in \bN$ is an infinitely differentiable function induced by $f$. The construction and properties of $\{e_k, k\in\bN \}$ are described in Remarks \ref{rem:ek:construction} and \ref{rem:ek:property}. 
\end{rem}

\begin{rem} \label{rem:ek:construction}
\begin{enumerate}[(i)]
\item 
Let $\nu$ be the corresponding measure of $\cF(f)$ and define 
\begin{equation} \label{191025_02}
\langle\phi,\psi\rangle_{\cH}:= (f,\phi*\tilde{\bar{\psi}})  ~= (\phi,f\ast\bar{\psi}),
\end{equation}
for $\phi,\psi\in \cS$.
Then, if $\ell\ll\nu$, the operator $\langle\cdot,\cdot\rangle_{\cH}$ is an inner product on $\cS$. To see this, it suffices to show that $\langle\psi,\psi\rangle_{\cH} = 0$ yields $\psi = 0$ ($\ell$-a.e.) since the operator $\langle\cdot,\cdot\rangle_{\cH}$ is sesquilinear on $\cS$. By Theorem \ref{191024_03} $(iii)$, we have
$$
\langle\psi,\psi\rangle_{\cH}=\int_{\bR^d}|\cF(\psi)(\xi)|^2\nu(d\xi),
$$
and thus $\langle\psi,\psi\rangle_{\cH}=0$ implies $\cF(\psi)=0$ ($\nu$-a.e.). 
Since $\ell\ll\nu$, $\cF(\psi)=0$ ($\ell$-a.e.) and thus $\psi=0$ ($\ell$-a.e.). 

\item
Define $\cH$ as the closure of $\cS$ under the norm $\|\cdot\|^2_{\cH}:=\langle\cdot,\cdot\rangle_{\cH}$. Then there exists a complete orthonormal basis $\{e_k:k\in\bN\}\subseteq \cS_R$ of $\cH$, since $\cS$ is dense in $L_2(\nu)$ and $L_2(\nu)$ is isomorphic to $\cH$. 

\item 
If $\ell\not\ll\nu$, the existence of a complete orthonormal basis of $\cH$ can be shown similarly by identification with an equivalence relation; see \cite{dalang1999extending,gel2014generalized}.

\end{enumerate}
\end{rem}

\begin{rem}\label{rem:ek:property}
\begin{enumerate}[(i)]
\item 
Since $f\in\cS'_{n}\cap\cS'_{nd}$ and $e_k\in\cS_R$, the real-valued function $v_k:=f\ast e_k$ is infinitely differentiable. 
\item 
The function $v_k$ is bounded. Indeed, by Theorem \ref{191024_03} $(iv)$, we have
\begin{equation*}
\begin{aligned}
|v_k(x)|&=|f\ast e_k(x)|=|(f,e_k(x-\cdot))|=|((1-\Delta)^{l/2}h,e_k(x-\cdot))|\\
&=|(h,(1-\Delta)^{l/2}e_k(x-\cdot))|\leq \sup_{x\in\bR^d}|h(x)| \| (1-\Delta)^{l/2}e_k \|_{L_1}.
\end{aligned}
\end{equation*}

\end{enumerate} 
\end{rem}


\vspace{2mm}
\subsection{Stochastic Banach spaces}

This section provides definitions and properties of H\"older spaces, Sobolev spaces (Bessel potential spaces) and stochastic Banach spaces; see \cite{grafakos2009modern,kry99analytic,krylov2008lectures}.

\begin{defn}
\begin{enumerate}[(i)]
\item Let $T<\infty$. For $\delta \in (0,1)$, set
\begin{equation*}
    \begin{gathered}
    \left|u\right|_{C([0,T])} := \sup_{t\in[0,T]}\left|u(t)\right|,\quad\left[u\right]_{C^{\delta}([0,T])} := \sup_{t,s\in[0,T],t\neq s}\frac{\left| u(t) - u(s)\right|}{|t-s|^\delta},\\
    |u|_{C^{\delta}([0,T])} := |u|_{C([0,T])} + [u]_{C^{\delta}([0,T])}.
    \end{gathered}
\end{equation*}

Let $C^{\delta}([0,T])$ be the set of bounded continuous functions on $[0,T]$ such that finite norm $|u|_{C^{\delta}([0,T])}<\infty$.

\item For $k = 0,1,2,\dots$, $\delta \in (0,1)$, and multi-index $\beta$, set
\begin{equation*}
    \begin{gathered}
    \left[u\right]_{C^k} := \sup_{|\beta| = k}\sup_{x\in\bR^d}\left|D^\beta u(x)\right|,\quad |u|_{C^k} := \sum_{j=0}^k\left[u\right]_{C^j},\\
\left[u\right]_{C^{k+\delta}} := \sup_{|\beta| = k}\sup_{x,y\in\bR^d,x\neq y}\frac{\left|D^\beta u(x) - D^\beta u(y)\right|}{|x-y|^\delta}, \quad
|u|_{C^{k+\delta}} := |u|_{C^k} + [u]_{C^{k+\delta}}.
    \end{gathered}
\end{equation*}

Let $C^{k+\alpha}$ be the set of bounded continuous functions on $\bR^d$ such that finite norm $|u|_{C^{k+\delta}}<\infty$.
\end{enumerate}
\end{defn}

\begin{defn}
For $p\in(1,\infty)$ and $\gamma \in \mathbb{R}$,
\begin{enumerate}[(i)]
\item  let $H_p^\gamma=H_p^\gamma(\bR^d)$ denote the class of tempered distributions $u$ on $\bR^d$ such that
$$ \| u \|_{H_p^\gamma} := \| (1-\Delta)^{\gamma/2} u\|_{L_p} = \| \cF^{-1}[ (1+|\xi|^2)^{\gamma/2}\cF(u)(\xi)]\|_{L_p}<\infty,
$$

\item let $H_p^\gamma(l_2) = H_p^\gamma(\bR^d;l_2)$ denote the class of $l_2$-valued tempered distributions $g=(g^1,g^2,\cdots)$ on $\bR^d$ such that
$$ \|g\|_{H_{p}^\gamma(l_2)}:= \| | (1-\Delta)^{\gamma/2} g|_{l_2}\|_{L_p} = \| |\cF^{-1}[ (1+|\xi|^2)^{\gamma/2}\cF(g)(\xi)]|_{l_2} \|_{L_p}<\infty.
$$	
\end{enumerate}
\end{defn}

\begin{rem}
\label{R}
It is well known (e.g., \cite{krylov2008lectures}) that for $\gamma\in (0,\infty)$, 
\begin{equation*}
(1-\Delta)^{-\gamma/2}u(x)=\int_{\bR^d}R_{\gamma}(x-y)u(y)dy,
\end{equation*}
where
\begin{equation*}
\label{ker'}
R_{\gamma}(x)=c(\gamma,d)|x|^{\gamma-d}\int_{0}^{\infty}t^{\frac{-d+\gamma-2}{2}}e^{-|x|^2t-\frac{1}{4t}}dt.
\end{equation*}
Observe that (e.g., \cite[Proposition 1.2.5.]{grafakos2009modern})
\begin{equation} 
\label{eqn 20.01.02}
R_\gamma(x) \leq N_1e^{-\frac{|x|}{2}}1_{|x|\geq2}+N_2 A_\gamma(x)1_{|x|<2},
\end{equation}
where 
\begin{equation*}
\begin{aligned}
A_{\gamma}(x)&:=(O(|x|^{\gamma-d+2})+|x|^{\gamma-d}+1)1_{0<\gamma<d}\\
  &\quad+(\log(2/|x|)+1+O(|x|^{2}))1_{\gamma=d}+(1+O(|x|^{\gamma-d}))1_{\gamma>d},
\end{aligned}
\end{equation*}
and $N_i~ (i = 1,2)$ depends only on $\gamma$ and $d$. 
\end{rem}\ \\


Below some useful facts on $H_p^\gamma$ are described.
\begin{lem}
\label{prop_of_bessel_space} Let $\gamma \in \bR$ and $p\in(1,\infty)$. 
\begin{enumerate}[(i)]
\item 
The spaces  $C_c^\infty$ and $\cS$ are dense in $H_{p}^{\gamma}$. 

\item
\label{sobolev-embedding} Let $\gamma - d/p = n+\nu$ for some $n=0,1,\cdots$ and $\nu\in(0,1]$. Then for any  $i\in\{ 0,1,\cdots,n \}$, we have
\begin{equation} \label{holder embedding}
| D^i u |_{C} + | D^n u |_{\cC^\nu} \leq N \| u \|_{H_{p}^\gamma},
\end{equation}
where $\cC^\nu$ is the Zygmund space.
\item 
Let 
\begin{equation*} \label{condition_of_constants_interpolation}
\begin{gathered}
\kappa\in[0,1],\quad p_i\in(1,\infty),\quad\gamma_i\in \bR,\quad i=0,1,\\
\gamma=\kappa\gamma_1+(1-\kappa)\gamma_0,\quad1/p=\kappa/p_1+(1-\kappa)/p_0.
\end{gathered}
\end{equation*}
Then we have
\begin{equation*}
\|u\|_{H^\gamma_{p}} \leq \|u\|^{\kappa}_{H^{\gamma_1}_{p_1}}\|u\|^{1-\kappa}_{H^{\gamma_0}_{p_0}}.
\end{equation*}
\end{enumerate}

\end{lem}

\begin{proof}
For (i) and (ii), see \cite{krylov2008lectures}. 
The multiplicative inequality (iii) follows from \cite{krein2002interpolation}.
\end{proof}

Recall that $(\Omega, \cF, \bP)$ is a complete probability space with a filtration $\{\cF_t\}_{t\geq0}$ satisfying the usual conditions and $\cP$ is the predictable $\sigma$-field related to $\cF_t$.

\begin{defn} \label{def:sto-banach}
For $\tau\leq T$, let us denote $\opar0,\tau\cbrk:=\{ (\omega,t):0<t\leq \tau(\omega) \}$.
\begin{enumerate}[(i)]
\item 
For $\tau\leq T$, {\it{Stochastic Banach spaces}} are defined by 
\begin{gather*}
\mathbb{H}_{p}^{\gamma}(\tau) := L_p(\opar0,\tau\cbrk, \mathcal{P}, d\bP \times dt ; H_{p}^\gamma),\\
\mathbb{H}_{p}^{\gamma}(\tau,l_2) := L_p(\opar0,\tau\cbrk,\mathcal{P}, d\bP \times dt;H_{p}^\gamma(l_2)),\\
U_{p}^{\gamma} :=  L_p(\Omega,\cF_0, d\bP ; H_{p}^{\gamma-2/p}).
\end{gather*}	
For convenience, we write $\bL_p(\tau):=\bH^{0}_{p}(\tau)$ and $\bL_p(\tau,l_2):=\bH^{0}_{p}(\tau,l_2)$.

\item \label{def:Hp-gamma}
The norm of each stochastic Banach space is defined in the usual way, e.g., 
\begin{equation} \label{norm}
\| u \|^p_{\bH^{\gamma}_{p}(\tau)} := \bE \int^{\tau}_0 \| u(t) \|^p_{H^{\gamma}_{p}}dt. 
\end{equation}

\end{enumerate}
\end{defn}

\begin{defn} \label{def_of_sol_1}
Let $\tau$ be a bounded stopping time and $u \in \bH_p^{\gamma}(\tau)$.
\begin{enumerate}[(i)]
\item 
We write $u\in\cH^{\gamma}_p(\tau)$ if $u_0\in U_{p}^{\gamma}$ and there exists $(f,g)\in
\bH_{p}^{\gamma-2}(\tau)\times\bH_{p}^{\gamma-1}(\tau,l_2)$ such that
\begin{equation*}
du = fdt+\sum_{k=1}^{\infty} g^k dw_t^k,\quad   t\in (0, \tau]\,; \quad u(0,\cdot) = u_0
\end{equation*}
in the sense of distributions, i.e., for any $\phi\in C_c^\infty$, the equality
\begin{equation} \label{def_of_sol_2}
(u(t,\cdot),\phi) = (u_0,\phi) + \int_0^t(f(s,\cdot),\phi)ds + \sum_{k=1}^{\infty} \int_0^t(g^k(s,\cdot),\phi)dw_s^k
\end{equation}
holds for all $t\in [0,\tau]$ almost surely.
In this case, we write
\begin{equation*}
\mathbb{D}u:= f,\quad \mathbb{S}u:=g.
\end{equation*}

\item
The norm of $\cH_{p}^{\gamma}(\tau)$ is defined by
\begin{equation*}
\| u \|_{\cH_{p}^{\gamma}(\tau)} :=  \| u \|_{\mathbb{H}_{p}^{\gamma}(\tau)} + \| \mathbb{D}u \|_{\mathbb{H}_{p}^{\gamma-2}(\tau)} + \| \mathbb{S}u \|_{\mathbb{H}_{p}^{\gamma-1}(\tau,l_2)} + \| u(0,\cdot) \|_{U_{p}^{\gamma}}.
\end{equation*}

\end{enumerate}
\end{defn}

\begin{defn}
Let $\tau$ be a stopping time. 
We write $u\in\cH_{p,loc}^{\gamma}(\tau)$ if there exists a sequence of bounded stopping times $\tau_n\uparrow \tau$ such that $u\in \cH_{p}^{\gamma}(\tau_n)$ for each $n$. 
We say $u = v$ in $\cH_{p,loc}^{\gamma}(\tau)$ if there are bounded stopping times $\tau_n\uparrow\tau$ such that $u = v$ in $\cH_{p}^{\gamma}(\tau_n)$ for each $n$.
For convenience, $\tau$ is omitted when $\tau = \infty$.
\end{defn}

\begin{rem} \label{schwartz}
Definition \ref{def_of_sol_1} can be defined with $\psi\in \cS$, instead of $\phi\in C_c^\infty$; see, e.g., \cite[Remark 3.4]{kry99analytic}. 
\end{rem}

\begin{rem} \label{equation_rewrite}
Main equation \eqref{origineq} can be understood in the sense of distribution. In other words, for any $\phi\in C_c^\infty$, the equality
\begin{equation*}
\begin{aligned}
(u(t,\cdot),\phi)= (u_0,\phi) &+ \int_0^t(a^{ij}(s,\cdot)u_{x^ix^j}(s,\cdot)+b^i(s,\cdot)u_{x^i}(s,\cdot)+c(s,\cdot)u(s,\cdot),\phi)ds \\
&+  \int_0^t\int_{\bR^d}\xi(s,x)|u(s,x)|^{1+\lambda}\phi(x) F(ds,dx)  
\end{aligned}
\end{equation*}
holds for all $t\in[0,\tau]$ almost surely. 
Due to Remark \ref{remark:representation} and \cite[Theorem 3.10]{kry99analytic}, if $u\in \cH_{p,loc}^\gamma$ for some $\gamma> 0$ and $p\geq2$, then we have
$$ \int_{0}^t\int_{\bR^d}\xi(s,x)|u(s,x)|^{1+\lambda}\phi(x)\,F(ds,dx)=\sum_{k=1}^{\infty}\int_0^t\int_{\bR^d} (f\ast e_k)(x)\xi(s,x)|u(s,\cdot)|^{1+\lambda} \phi(x)dxdw^k_s.
$$
Thus, we prove that the equation
\begin{equation*}
du=(a^{ij}u_{x^ix^j}+b^iu_{x^i}+cu)dt+\sum_{k=1}^{\infty}\xi|u|^{1+\lambda}(f\ast e_k)dw^k_t,
\quad (t,x) \in (0,\infty) \times \bR^d ; \quad u(0,\cdot)=u_0,
\end{equation*}
has a unique  solution $u\in\cH_{p,loc}^{\gamma}$.
\end{rem}

Below embedding theorems for stochastic Banach spaces are introduced.

\begin{thm} \label{embedding}
Let $\tau\leq T$ be a bounded stopping time and $\gamma \in \bR$.
\begin{enumerate}[(i)]
\item \label{embedding:1}
For any $p\in [2,\infty)$, $\cH_p^{\gamma}(\tau)$ is a Banach space with the norm $\|\cdot\|_{\cH_p^{\gamma}(\tau)}$.

\item \label{h_embedding} 
If $p\in(2,\infty)$ and $\frac{1}{p}< \alpha < \beta < \frac{1}{2}$, then for any $u\in\cH_p^{\gamma}(\tau)$, we have 
$$
u\in C^{\alpha-\frac{1}{p}}([0,\tau];H_p^{\gamma-2\beta})\quad(a.s.)
$$  and
\begin{equation}\label{holder_embedding}
\bE|u|^p_{C^{\alpha-\frac{1}{p}}([0,\tau];H_p^{\gamma-2\beta})}\leq N(d,p,\alpha,\beta,T)\|u\|^p_{\cH_p^{\gamma}(\tau)}.
\end{equation}

\item \label{embedding:3}
If $p=2$, then inequality \eqref{holder_embedding} holds for $\alpha = \beta = 1/p = 1/2$, i.e., $u\in C([0,\tau];H_2^{\gamma-1})$ (a.s.) and
\begin{equation} \label{holder_embedding_2}
\bE\sup_{t\leq\tau}\|u\|^2_{H_2^{\gamma-1}}  \leq N(d,T)\|u\|^2_{\cH_2^{\gamma}(\tau)}.
\end{equation}

\item \label{gronwall-type} 
If $p\in[2,\infty)$ and $\tau \equiv T$, then for  any $t\leq T$,
\begin{equation} \label{stochastic_gronwall}
\|u\|^p_{\bH^{\gamma-1}_{p}(t)}\leq N(d,p,T)\int^t_0 \|u\|^p_{\cH^{\gamma}_{p}(s)}\,ds.
\end{equation}
\end{enumerate}
\end{thm}

\begin{proof}
Theorem \ref{embedding} \eqref{embedding:1}-\eqref{embedding:3} are proved in \cite{kry99analytic}; 
for \eqref{embedding:1}, see Section 3; 
for \eqref{h_embedding} and \eqref{embedding:3}, see Section 7. 
In the case of \eqref{gronwall-type}, use \eqref{norm} and inequalities \eqref{holder_embedding} and \eqref{holder_embedding_2}.
\end{proof}

 We introduce H\"older embedding theorem for solution space $\cH_p^\gamma(\tau)$.
\begin{corollary} \label{s.t. holder}
Let $\tau\leq T$ be a bounded stopping time and $p\in(2,\infty)$, $\gamma\in(0,\frac{1}{2})$, $\alpha,\beta\in(\frac{1}{p},\frac{1}{2})$ such that
\begin{equation}
\label{alpha,beta}
\frac{1}{p}<\alpha<\beta<\frac{\gamma}{2}-\frac{d}{2p}.
\end{equation}
Then for any $\delta \in [0,\gamma-2\beta-\frac{d}{p})$, we have 
\begin{equation}
\label{space-time_holder}
\bE|u|^p_{C^{\alpha-\frac{1}{p}}([0,\tau];C^{\gamma-2\beta-\frac{d}{p}-\delta}(\bR^d) )}\leq N(d,p,\alpha,\gamma,T)\|u\|_{\cH_p^{\gamma}(\tau)}^p.
\end{equation}
\end{corollary}

\begin{proof}
Note that Lemma \ref{prop_of_bessel_space} \eqref{sobolev-embedding} implies 
\begin{equation}\label{ineq:proof:embedding:cor}
|u|_{C^{\gamma-2\beta-\frac{d}{p}-\delta}(\bR^d)} \leq N\| u \|_{H_p^{\gamma-2\beta-\delta}}\leq N\| u \|_{H_p^{\gamma-2\beta}}.
\end{equation}
Applying inequality \eqref{ineq:proof:embedding:cor} 
to the definition of the norm $|\cdot|_{C^{\alpha-\frac{1}{p}}([0,\tau];C^{\gamma-2\beta-\frac{d}{p}-\delta}(\bR^d) )}$ and employing inequality \eqref{holder_embedding} prove \eqref{space-time_holder}. The corollary is proved.
\end{proof}
\vspace{2mm}
\mysection{Main results}
\label{sec:main}

In this section, we describe some major results on a solution to equation 
\begin{equation}
\label{targeteq}
du =(a^{ij}u_{x^ix^j}+b^iu_{x^i}+cu)dt+\sum_{k=1}^{\infty}\xi|u|^{1+\lambda}(f\ast e_k)dw^k_t,
\quad (t,x) \in (0,\infty) \times \bR^d;
\quad u(0,\cdot) = u_0(\cdot),
\end{equation}
where $\lambda \geq 0$ and $u_0\geq0$. The existence, uniqueness, and regularity of a solution in $\cH_{p,loc}^\gamma$ are proved and H\"older regularity of the solution is achieved. In the case of $\lambda = 0$, we consider more general equation than \eqref{targeteq}; see equation \eqref{targeteq:lipschitz}. 
Investigation on equation \eqref{targeteq:lipschitz} is employed to prove the super-linear case; see Section \ref{sec:main:superlinear}. 
Specifically, it turns out that $u\geq 0$, when $\lambda>0$, by maximum principle.
Then by taking integration with respect to $x$ in \eqref{targeteq} (at least formally), one can check that $\| u(t,\cdot) \|_{L_1}$ is a continuous local matingale. Since $\| u(t,\cdot) \|_{L_1}$ is nonnegative, its paths are bounded almost surely. 
If $\hat{\xi} := \xi|u|^{1+\lambda}/u$ ($0/0:=0$) agrees with the summability condition of Theorem \ref{case=0}, the unique solvability of \eqref{targeteq} can be obtained.

It should be remarked that more regularity of the solution can be obtained when the colored noise $F$ is simple in some sense; for example, the covariance $f$ is smooth or the spatial dimension $d$ is small. 
Furthermore, we discuss the range of $\lambda$, a sufficient condition for the unique solvability, in relation to $f$ and $d$.

\subsection{The Lipschitz case}
\label{sec:main:lipschitz}

In this section, consider an SPDE with Lipschtiz diffusion coefficients; for a bounded stopping time $\tau\leq T$,
\begin{equation}
\label{targeteq:lipschitz}
du =(a^{ij}u_{x^ix^j}+b^iu_{x^i}+cu)dt+\sum_{k=1}^{\infty}\xi\, h(u)\,(f\ast e_k)dw^k_t,
\quad (t,x) \in (0,\tau) \times \bR^d;
\quad u(0,\cdot) = u_0(\cdot),
\end{equation}
where $a^{ij},b^i,c$ and $\xi$ are $\cP\times\cB(\bR^d)$-measurable functions and $h(u)$ satisfies Lipschitz condition in $u$; for more details, see Assumptions \ref{ass coeff''}-\ref{assumption_for_h}. An estimate, existence, uniqueness, and regularity of a solution to \eqref{targeteq:lipschitz} in $\cH_p^\gamma$ are provided, and H\"older regularity of the solution is obtained; see Theorem \ref{case=0}. 
In particular, these results cover the case $\lambda= 0$ in \eqref{targeteq}. 

\vspace{4mm}

\begin{assumption} \label{ass coeff''}
\begin{enumerate}[(i)]
\item \label{ass coeff'' 1} 
The coefficients $a^{ij}$, $b^{i}$, $c$ are $\cP\times \cB(\bR^d)$-measurable.

\item \label{ass coeff'' 2} 
There exist constants $\kappa_0,K>0$ such that 
\begin{equation}
\label{ellipticity''}
 \kappa_0|\eta|^2\leq a^{ij}(t,x)\eta^i\eta^j\leq K|\eta|^2\quad \forall \omega,t,x,\eta
\end{equation}
and
\begin{equation}
\label{bounded''}
|a^{ij}(t,\cdot)|_{ C^2(\bR^d)}+|b^i(t,\cdot)|_{ C^2(\bR^d)} + |c(t,\cdot)|_{ C^2(\bR^d)}\leq K\quad \forall \omega,t.
\end{equation}
\end{enumerate}
\end{assumption}

\begin{assumption}[$\tau,\sigma$]\label{assumption_for_xi}
The  coefficient $\xi$ is $\mathcal{P}\times\mathcal{B}(\bR^d)$-measurable, and 
there exists a constant $K>0$ such that 
\begin{equation} \label{s_bound_of_xi}
\| \xi(t,\cdot) \|_{L_{\sigma}}\leq K\quad \forall\, \omega \in \Omega, \,\, t\in [0,\tau]\cap [0,\infty).
\end{equation}
\end{assumption}

\begin{assumption}[$\tau,p$] \label{assumption_for_h}
The function $h=h(\omega,t,x,u)$ is $\mathcal{P}\times\mathcal{B}(\bR^d)\times\mathcal{B}(\mathbb{R})$-measurable, $h(0)=h(\omega,t,x,0)\in\bL_p(\tau)$, and 
there exists a constant $K>0$ such that for any $\omega\in\Omega$,  $t\in[0,\tau]$, $u,v\in \bR$ and $x\in\bR^d$,
\begin{equation} \label{eq_assumption_for_h}
| h(t,x,u)-h(t,x,v) |  \leq K| u-v |.
\end{equation}

\end{assumption}

\begin{rem}
Obviously, $h(u) = u$ or $h(u) = |u|$ satisfy Assumption \ref{assumption_for_h}.
\end{rem}

\vspace{2mm}

Now, we describe the results on the Lipschitz case \eqref{targeteq:lipschitz}. 

\vspace{2mm}

\begin{thm} \label{case=0}
Let $\tau\leq T$ be a bounded stopping time.
Assume that $f$, $d$, $\gamma$, $s$, $p$, and $\sigma$ satisfy one of the following conditions;
\begin{enumerate}[(i)]
\item $f\in\cS'_n\cap\cS'_{nd}$, $d=1$, $\gamma\in[0,1/2)$, $s\in\left(\frac{1}{1-2\gamma},\infty\right]$, $p\in\left[\frac{2s}{s-1},\infty\right)$ $\left(\frac{2\infty}{\infty}:=2 \right)$, and $\sigma = 2s$;
\item $f\in\cS'_n\cap\cS'_{nd}$, $d\geq 2$, $\gamma\in[0,1)$, $s\in\left(\frac{d}{1-\gamma},\infty\right]$, $p\in\left[\frac{2s}{s-1},\infty\right)$ $\left(\frac{2\infty}{\infty}:=2 \right)$, $\sigma = 2s$, and 
\begin{equation}
\label{20.12.14.11.01}
\int_{|x|<1}|x|^{\frac{s(1-\gamma-d)}{s-1}}\mu(dx)<\infty\quad  \left((1-\gamma-d)\frac{\infty}{\infty}:= 1-\gamma-d \right),
\end{equation}
where $\mu$ is the corresponding measure of $f$;
\item $f\in \cS_n'\cap C_{nd}$, $d\geq 1$, $\gamma\in[0,1)$, $s\in\left(\frac{d}{1-\gamma},\infty\right]$, $p\in\left[\frac{2s}{s-1},\infty\right)$ $\left(\frac{2\infty}{\infty}:=2 \right)$, and $\sigma = s$.
\end{enumerate} 
Suppose that Assumptions  \ref{ass coeff''}, \ref{assumption_for_xi}$(\tau,\sigma)$ and \ref{assumption_for_h}$(\tau,p)$ hold. Then, for $u_0 \in U_p^\gamma$, equation \eqref{targeteq:lipschitz} with initial data $u(0,\cdot) = u_0$ has a unique solution $u$ in $\cH_p^{\gamma}(\tau)$. The solution satisfies
\begin{equation}
\label{0_estimate}
\|u\|_{\cH_p^{\gamma}(\tau)}\leq N(\|h(0)\|_{\bL_p(\tau)}+ \|u_0\|_{U_p^{\gamma}}),
\end{equation}
where $N=N(T, d,p,\gamma, \kappa_0,K)$. Furthermore, for any $\alpha,\beta$ and $\delta$ satisfying
\begin{equation*}
\frac{1}{p}<\alpha<\beta<\frac{\gamma}{2}-\frac{d}{2p}, \quad 0\leq \delta <\gamma-2\beta-\frac{d}{p},
\end{equation*}
we have
\begin{equation*}
| u |_{C^{\alpha-\frac{1}{p}}([0,T];C^{\gamma-2\beta-\frac{d}{p}-\delta}(\bR^d))}<\infty
\end{equation*}
for all $T<\infty$ (a.s.).
\end{thm}
\begin{proof}
See Section \ref{sec:lambda=0}.
\end{proof}

\begin{rem}

It should be remarked that in condition (i) of Theorem \ref{case=0}, conditions on $\mu$ is not imposed, such as \eqref{20.12.14.11.01}. Indeed, in the proof of Theorem \ref{case=0}, to obtain \eqref{0_estimate} with finite constant $N$, we need to show
\begin{equation}
\label{near_zero_mu_integrable_lipschitz_rmk}
\int_{|x|\leq 1} \left|\left(R^{\frac{s}{s-1}}_{1-\gamma}\ast R^{\frac{s}{s-1}}_{1-\gamma}\right)(x)\right| \mu(dx)<\infty,
\end{equation}
where $\mu$ is the corresponding measure of $f$ and $R_{1-\gamma}$ is the function introduced in Remark \ref{R}; see \eqref{near_zero_mu_integrable_lipschitz}.
In the case of $(i)$, since $R_{1-\gamma}\in L_{2r}(\bR)$, we have
\begin{equation}
\label{near_zero_mu_integrable_lipschitz_rmk_2}
\sup_{x\in\bR}\left|\left(R^{\frac{s}{s-1}}_{1-\gamma}\ast R^{\frac{s}{s-1}}_{1-\gamma}\right)(x)\right|\leq \|R_{1-\gamma}\|_{L_{2r}(\bR)}^{ 2r}<\infty.
\end{equation}
Thus, \eqref{191024_02} and \eqref{near_zero_mu_integrable_lipschitz_rmk_2} imply \eqref{near_zero_mu_integrable_lipschitz_rmk}; see \eqref{near_zero_mu_integrable_lipschitz_sup}. 

On the other hand, in the case of $(ii)$, by Lemma \ref{decay}, we only have
$$\left| \left(R^{\frac{s}{s-1}}_{1-\gamma}\ast R^{\frac{s}{s-1}}_{1-\gamma}\right)(x) \right| \leq N|x|^{\frac{s(1-\gamma-d)}{s-1}}.
$$
Therefore, condition \eqref{20.12.14.11.01} is required. For more information, we refer the reader to Remark 4 and Theorem 6 of \cite{ferrante2006spdes}.
\end{rem}

\begin{rem}
Note that Theorem \ref{case=0} holds for $s = \infty$ and $h(u) = |u|$, which is equivalent to equation \eqref{targeteq}
on $t\leq \tau$ with $\lambda = 0$.
\end{rem}

\vspace{4mm}
\subsection{The super-linear case}
\label{sec:main:superlinear}

This section provides unique solvability of \eqref{targeteq} and regularities of a solution to equation \eqref{targeteq}. In addition, maximal H\"older regularity is obtained.

Before we state the major results, 
assumptions on coefficients are introduced. Note that the coefficients can be dependent on $\omega,t,$ and $x$.

\begin{assumption}
\label{assumption_xi}
The coefficient $\xi$ is $\cP\times \cB(\bR^d)$-measurable and there exists $K > 0$ such that 
$$  \| \xi(t,\cdot) \|_{L_{\infty}}\leq K\quad\forall\omega,t.
$$
\end{assumption}

\begin{defn} 
\label{def:gamma0,1}
For $d \in \bN$, let
\begin{gather*}
\gamma_0=\gamma_0(d,\lambda):=\frac{1}{2}-2d\left(\lambda-\frac{1}{4d}\right)1_{\{\lambda>\frac{1}{4d}\}}, \\
\gamma_1=\gamma_1(d,\lambda):= \frac{1}{2}-d\left(\lambda-\frac{1}{2d}\right)1_{\{\lambda>\frac{1}{2d}\}}.
\end{gather*}
\end{defn}

Now, the unique solvability of equation \eqref{targeteq} in $\cH_{p,loc}^\gamma$ is presented. Besides, $L_p$ and H\"older regularities of the solution are obtained.

\begin{thm}
\label{main}
Suppose Assumptions \ref{ass coeff''} and \ref{assumption_xi} hold. 
The constants $\gamma_0$ and $\gamma_1$ are defined in Definition \ref{def:gamma0,1}. 
Assume that f, $d$, $\lambda$, $\gamma$, and $p$ satisfy one of the following conditions:

\begin{enumerate}[(i)]
\item  \label{main_thm_first_condi}
$f\in \cS'_{n}\cap\cS'_{nd}$, $d = 1$,  $0\leq\lambda<\frac{1}{2}$, $0<\gamma<\frac{1}{2}-\lambda$, and $p\in \left( \frac{3}{\gamma},\infty \right)$;

\item \label{main_thm_second_condi}
$f\in \cS'_{n}\cap\cS'_{nd}$, $d\geq2$, $0\leq\lambda<\frac{1}{2d}$, $0 < \gamma < \gamma_0$,  $p \in\left(\frac{d+2}{\gamma},\infty\right)$, and
\begin{equation} \label{eq.08.10.}
\int_{|x|<1}|x|^{\frac{1-\gamma-d}{1-2\lambda}}\mu(dx)<\infty,
\end{equation}
where $\mu$ is the corresponding measure of $f$;

\item \label{main_thm_third_condi}
$f\in \cS'_{n}\cap C_{nd}$, $d\geq1$, $0\leq \lambda<\frac{1}{d}$, $0 < \gamma < \gamma_1$, and $p \in\left(\frac{d+2}{\gamma},\infty\right)$.
\end{enumerate}
Then, for nonnegative $u_0$ is in $ U_p^{\gamma}\cap L_1(\Omega;L_1(\bR^d))$, equation \eqref{targeteq} with initial data $u(0,\cdot) = u_0$ has a unique solution $u$ in $\cH_{p,loc}^{\gamma}$.
In addition, for any $\alpha,\beta$ and $\delta$ satisfying
\begin{equation} \label{condi_const_Holder}
\frac{1}{p}<\alpha<\beta<\frac{\gamma}{2}-\frac{d}{2p}, \quad 0\leq \delta <\gamma-2\beta-\frac{d}{p},
\end{equation}
we have
\begin{equation}
\label{regularity of solution}
| u |_{C^{\alpha-\frac{1}{p}}([0,T];C^{\gamma-2\beta-\frac{d}{p}-\delta}(\bR^d))}<\infty
\end{equation}
for all $T<\infty$ (a.s.).
\end{thm}
\begin{proof}
In the case of $\lambda = 0$, we use Theorem \ref{case=0}. Let $\tau$ be a bounded stopping time. 
Note that the conditions in Theorem \ref{main} with $\lambda = 0$ implies those in Theorem \ref{case=0} for $s = \infty$, and thus there exists 
$u = u_\tau(t,x)\in \cH_{p}^\gamma(\tau)$. 
Similarly, for a stopping time $\tilde\tau$ such that $\tau\leq \tilde\tau$, we can find $\hat{u}\in \cH_{p}^\gamma(\tilde\tau)$. From the uniqueness of solutions in $\cH_p^\gamma(\tau)$, we have $\hat{u} = u$ in $\cH_p^\gamma(\tau)$. Thus, by the definition of $\cH_{p,loc}^\gamma$, the Lipschitz case is proved. 
It should be remarked that the proof of the Lipschtiz case does not require the condition nonnegativity of the initial data $u_0$.

For the case of $\lambda > 0$, see Section \ref{sec:superlinear}. 
\end{proof}

\begin{rem}
The range of $\lambda$ in each condition of Theorem \ref{main} is a sufficient condition for the unique solvability. More details on the range of $\lambda$ are discussed in Section \ref{sec:superlinear}; see Remark \ref{rmk:range_of_lambda}.
\end{rem}

\begin{rem}
Due to the assumptions on $\gamma$ and $p$, there always exist $\alpha$ and $\beta$ satisfying inequality \eqref{condi_const_Holder}. Besides, the assumptions make the number $\gamma-2\beta-\frac{d}{p}-\delta$ in \eqref{regularity of solution} positive.
\end{rem}

\begin{rem}
By \eqref{regularity of solution}, H\"older regularity of the solution is obtained and it varies according to the choice of $\alpha,\beta,\gamma$ and $p$. For example, with sufficiently small $\ep>0$, if $\alpha$ and $\beta$ close enough to $1/p$ and $\delta = 0$, then
$$ \sup_{t\in[0,T]} | u(t,\cdot) |_{C^{\gamma-\frac{d+2}{p}-\ep}} <\infty,
$$
for all $T<\infty$ almost surely.
On the other hand, if $\alpha$ and $\beta$ close enough to $\frac{\gamma}{2}-\frac{d}{2p}$ and $\delta = 0$, then
$$ \sup_{x\in\bR^d} |u(\cdot,x) |_{C^{\frac{\gamma}{2}-\frac{d+2}{2p}-\ep}} <\infty,
$$
for all $T<\infty$ almost surely. Note that the solution $u$ depends not only on $(\omega,t,x)$ but also on $(\gamma,p)$.
\end{rem}

\begin{example}[\textbf{Riesz kernel}]

Let $f_{\alpha}(x)=|x|^{-\alpha}$ for some $\alpha\in (0,d)$. Since the covariance $f_{\alpha}\in\cS_n'\cap \cS'_{nd}$, Theorem \ref{main} can be applied.
\begin{enumerate}[(i)]
\item
If $d=1$ and $\lambda\in [0,1/2)$, Theorem \ref{main} holds for any $\alpha\in(0,1)$.

\item \label{ex:riesz:2}
If $d\geq2$ and $\lambda\in [0,\frac{1}{2d})$, Theorem \ref{main} holds for any 
\begin{equation} \label{example:condi}
\alpha \in \left(0,\frac{1-\gamma-2\lambda d}{1-2\lambda}\right).
\end{equation}

To see \eqref{ex:riesz:2}, note that the corresponding measure of $f_{\alpha}$ is $\mu(dx) = |x|^{-\alpha}dx$ (e.g. \cite{gel2014generalized}). 
To satisfy inequality \eqref{eq.08.10.}, we need $\frac{1-\gamma-d}{1-2\lambda}-\alpha+d>0$, which implies inequality \eqref{example:condi}.

\end{enumerate}
\end{example}

\begin{example}[\textbf{Gaussian kernel}] Let $f(x)=e^{-c|x|^2}$
. Since $f\in\cS'_n\cap C_{nd}$, Theorem \ref{main} (iii) holds for any Gaussian kernels.
\end{example}
\vspace{3mm}

To obtain maximal H\"older regularity of the solution by extending Theorem \ref{main}, the following theorem is required. 

\begin{thm}
\label{p_consistent}
Let $\gamma_0$, $\gamma_1$, $d$, $\lambda$, $\gamma$, $p$, $f$, and $u_0$ satisfy the conditions in Theorem \ref{main} and $u\in\cH_{p,loc}^\gamma$ be the solution. Then if $q > p$ and $u_0\in U_q^{\gamma}\cap L_1(\Omega;L_1(\bR^d))$, the solution $u$ also belongs to $\cH_{q,loc}^\gamma$. 
\end{thm}
\begin{proof}
See Section \ref{sec:superlinear}.
\end{proof}

Consequently, combining Theorems \ref{main} and \ref{p_consistent} leads to maximal H\"older regularity of the solution to \eqref{targeteq}:

\begin{corollary}
\label{190916_01}
Let $\gamma_0$ and $\gamma_1$ be defined in Definition \ref{def:gamma0,1}. 
Suppose that $d$, $\lambda$, $\gamma^\ast$, and $f$ satisfy one of the following conditions;

\begin{enumerate}[(i)]
\item  \label{max_holder_embedding_1}
$d = 1$, $0\leq\lambda<\frac{1}{2}$, and $\gamma^\ast=\frac{1}{2}-\lambda$;

\item  \label{max_holder_embedding_2}
$d\geq2$, $0\leq\lambda<\frac{1}{2d}$,  $0 < \gamma^{*} < \gamma_0$, and
\begin{equation*} 
\int_{|x|<1}|x|^{\frac{1-\gamma^*-d}{1-2\lambda}}\mu(dx)<\infty,
\end{equation*}
where $\mu$ is the corresponding measure of $f$;
 
\item  \label{max_holder_embedding_3}
$d\geq1$, $0\leq \lambda<\frac{1}{d}$,  $\gamma^\ast = \gamma_1$, and $f\in\cS'_n\cap C_{nd}$. 
\end{enumerate}
Suppose nonnegative $u_0$ is in $U_p^{\gamma^*}\cap L_1(\Omega;L_1(\bR^d))$ for any $p \in\left(\frac{d+2}{\gamma^*},\infty\right)$.  
Then 
for any $\gamma \in (0, \gamma^*)$ and $p \in\left(\frac{d+2}{\gamma},\infty\right)$, 
equation \eqref{targeteq} with $u(0,\cdot) = u_0$ has a unique solution $u_p$ in 
$\cH_{p,loc}^{\gamma}$.
In addition, for fixed $\gamma\in (0,\gamma^*)$, 
all the solutions $u_p$ coincide with $u$ for all $p \in\left(\frac{d+2}{\gamma},\infty\right)$. 
Furthermore, 
for any $T>0$ and  $0<\varepsilon<\frac{\gamma^\ast}{2}$, almost surely,
\begin{equation}
\label{Holder_d=1}
\sup_{t\in[0,T]}|u(t,\cdot)|_{C^{\gamma^*-\varepsilon}(\bR^d)}<\infty\quad \text{and} \quad\sup_{x\in\bR^d}|u(\cdot,x)|_{C^{\frac{\gamma^*}{2}-\varepsilon}([0,T])}<\infty.
\end{equation}

\end{corollary}

\begin{proof}
To see the first assertion, we use Theorem \ref{main}. Let $\gamma \in (0, \gamma^*)$ and $p > \frac{d+2}{\gamma}$. Since $U_p^{\gamma^*}\subseteq U_p^{\gamma}$, we have $u_0 \in U_p^{\gamma}\cap L_1(\Omega;L_1(\bR^d))$. Besides, in the case of \eqref{max_holder_embedding_2}, observe that $|x|^{\frac{1-\gamma-d}{1-2\lambda}}1_{|x|<1}\in L_1(\mu)$ for any $\gamma\in(0,\gamma^*]$. Thus, Theorem \ref{main} implies the first assertion. The second assertion follows from Theorem \ref{p_consistent}. For the last assertion, take $\gamma$ close enough to $\gamma^*$ and $p>\frac{d+2}{\gamma}$. By \eqref{regularity of solution}, for any $\alpha, \beta$ such that $\frac{1}{p}<\alpha<\beta<\frac{\gamma}{2}-\frac{d}{2p}$, we have 
$$
| u |_{C^{\alpha-\frac{1}{p}}([0,T];C^{\gamma-2\beta-\frac{d}{p}}(\bR^d))}<\infty \quad \forall T < \infty~ (a.s.).
$$
Fix $T>0$ and $\ep\in(0,\frac{\gamma^*}{2})$. By taking $\alpha,\beta$ close to $1/p$  and sufficiently large $p$, H\"older regularity with respect to spatial variable is obtained. On the other hand, by choosing $\alpha,\beta$ close to $\frac{\gamma}{2}-\frac{d}{2p}$  and sufficiently large $p$, we get H\"older regularity in time variable. Therefore, we have \eqref{Holder_d=1}. The corollary is proved.
\end{proof}

\mysection{Proof for case with the general Lipschitz diffusion coefficient } 
\label{sec:lambda=0}

In this section, we provide a proof of Theorem \ref{case=0}, which presents existence, uniqueness, and $L_p$-estimate of a solution to Lipschitz case \eqref{targeteq:lipschitz}. For readers' convenience, we recall equation \eqref{targeteq:lipschitz}:
\begin{equation} \label{targeteq:lipschitz:1}
du=(a^{ij}u_{x^ix^j}+b^iu_{x^i}+cu) \,dt+\sum_{k=1}^{\infty}\xi h(u)(f\ast e_k)dw^k_t, 
\quad 0<t\leq \tau;
\quad u(0,\cdot)=u_0,
\end{equation}
where $h(u)=h(\omega,t,x,u)$ satisfies Lipschitz condition in $u$; see Assumption \ref{assumption_for_h}($\tau,p$).

To prove Theorem \ref{case=0}, we employ Theorem \ref{key} that provides a unique solvability and estimate of a solution to the linear equation
\begin{equation}
\label{quasi-linear1}
du=(a^{ij}u_{x^ix^j}+b^iu_{x^i}+cu)\, dt+ \sum_{k=1}^{\infty}g^k(u)dw^k_t, 
\quad 0<t\leq \tau;
\quad u(0,\cdot)=u_0,
\end{equation}
where $\tau\leq T$ is a bounded stopping time. 
Here, the diffusion coefficient $g^k(u)=g^k(\omega,t,x,u)$ satisfies Assumption \eqref{ass g} $(\tau)$. 
To show that the diffusion coefficients for the Lipschitz case \eqref{targeteq:lipschitz:1} satisfy Assumption \eqref{ass g} $(\tau)$, we employ Lemma \ref{4_key}, which presents an estimate for the coefficients. 
Detailed proof for Theorem \ref{case=0} is described at the end of this section. 
\\

\begin{assumption}[$\tau$] \label{ass g}
\begin{enumerate}[(i)]
\item  For any $u\in H_p^{\gamma}$, the function $g(t,x,u)$ is a $H_p^{\gamma-1}(l_2)$-valued predictable function, and $g(0)\in \bH_p^{\gamma-1}(\tau,l_2)$, where $\tau$ is a stopping time.

\item  For any $\varepsilon>0$, there exists a constant $N_{\varepsilon}$ such that
\begin{equation}
\label{g}
\|g(t,u)-g(t,v)\|_{H_p^{\gamma-1}(l_2)}\leq \varepsilon \|u-v\|_{H_p^{\gamma}}+N_{\varepsilon}\|u-v\|_{H_p^{\gamma-1}},
\end{equation}
for any $u,v\in H_p^{\gamma}$, $\omega\in\Omega$ and $t\in[0,\tau]$.
\end{enumerate}
\end{assumption}

\begin{thm}
\label{key}
Let $\tau\leq T$ be a bounded stopping time, $|\gamma|\leq 1$, and $p\in[2,\infty)$. Suppose that Assumptions \ref{ass coeff''}, and \ref{ass g}($\tau$) hold. Then for any $u_0\in U_p^{\gamma}$,
equation \eqref{quasi-linear1} with $u(0,\cdot) = u_0$ has a unique solution $u$ in $\cH_p^{\gamma}(\tau)$. The solution satisfies
\begin{equation}
\label{qlestimate}
\|u\|_{\cH_p^{\gamma}(\tau)} \leq N(\|g(0)\|_{\bH_p^{\gamma-1}(\tau)}+\|u_0\|_{U_p^{\gamma}}),
\end{equation}
where $N=N(d,p,T, \kappa_0,K,\gamma)$.
\end{thm}

\begin{proof} 
Note that by \cite[Theorem 5.1]{kry99analytic} the claim holds if $\tau$ is a constant. 
Otherwise, consider
\begin{equation*}
\hat{g}(t,u) :=1_{t\leq\tau}g(t,u).
\end{equation*}
Since $\hat{g}$ satisfies \eqref{g} for all $t\leq T$, applying \cite[Theorem 5.1]{kry99analytic} leads to $u\in\cH_p^{\gamma+1}(T)$, which satisfies equation \eqref{quasi-linear1} and estimate \eqref{qlestimate} with $\hat{g}$, in place of $g$. Then, one can check that  $u\in\cH_p^{\gamma+1}(\tau)$ and it satisfies equation \eqref{quasi-linear1}. The existence is proved.


To prove the uniqueness, assume that $v\in\cH_p^{\gamma}(\tau)$ is a solution to the Cauchy problem \eqref{quasi-linear1} with initial data $u_0$. From \cite[Theorem 5.1]{kry99analytic}, there exists $\hat v\in\cH_p^{\gamma}(T)$,  the unique solution of equation
\begin{equation}
\label{unique}
\begin{aligned}
d\hat{v}=(a^{ij}\hat{v}_{x^ix^j}+b^{i}\hat{v}_{x^i}+c\hat{v})\, dt+\sum_{k=1}^{\infty}\hat{g}^k(v)dw^k_t, \quad t>0;\quad \hat{v}(0,\cdot)=u_0.
\end{aligned}
\end{equation}
Then $\varphi := \hat v - v$ satisfies equation
\begin{equation*}
\begin{aligned}
d\varphi=(a^{ij}\varphi_{x^ix^j}+b^{i}\varphi_{x^i}+c\varphi)\,dt, \quad 0<t\leq\tau;\quad \varphi(0,\cdot)=0.\\
\end{aligned}
\end{equation*}
Therefore, $\hat v = v$ by the uniqueness of a solution to deterministic parabolic equations (e.g. \cite[Theorem 5.1]{ladyvzenskaja1988linear}).
Consequently, $\hat{g}^k(v)$ in \eqref{unique} can be replaced with $\hat{g}^k(\hat{v})$. Similarly, there exists $\hat u\in\cH_p^\gamma(T)$ such that $u = \hat u$ and
$$d\hat{u}=(a^{ij}\hat{u}_{x^ix^j}+b^{i}\hat{u}_{x^i}+c\hat{u})\, dt+\sum_{k=1}^{\infty}\hat{g}^k(\bar u)dw^k_t, \quad t>0;\quad \hat{u}(0,\cdot)=u_0.
$$ 

From the uniqueness result in $\cH_p^{\gamma}(T)$, we have $\hat{u} = \hat{v}$ in $\cH_p^{\gamma}(T)$. Thus, $u=v$ in $\cH_p^{\gamma}(\tau)$. The theorem is proved.
\end{proof}\

Lemma \ref{decay} describes behavior of $R_{1-\gamma}^p\ast R_{1-\gamma}^p$.

\begin{lem}
\label{decay}
Let $g$ and $h$ be nonnegative, radial  and decreasing functions. Assume that $R(x) \in L_1(\bR^d)$ satisfies
\begin{equation*}
R(x) =  O(g(x)) \quad\mbox{as}\quad |x|\to\infty,
\end{equation*}
and 
\begin{equation*}
|R(x)| \leq Nh(x)\quad (a.e.),
\end{equation*}
where $N$ is independent of $x$.
Then  
\begin{equation} \label{behavior of RR}
(R\ast R)(x) =  O(g(x/3))\quad\mbox{as}\quad |x|\to\infty,
\end{equation}
and
\begin{equation} \label{behavior of RR_2}
|(R\ast R)(x)| \leq N' h(x/2) \quad (a.e.),
\end{equation}
where $N'$ is independent of $x$.

Furthermore, if  $h(x)=|x|^{-\alpha}$ for some $0<\alpha<d/2$, then $R\ast R \in L_{\infty}(\bR^d)$.
\end{lem}
\begin{proof} 
Since $R(x)=O(g(x))$ as $|x|\to\infty$, there exist $c,M>0$ such that  $|R(x)|\leq Mg(x)$ for all $|x|>c$. Then for $|x|>3c$, we have
\begin{equation*}
\begin{aligned}
|(R\ast R)(x)| &\leq \int_{|y|>\frac{|x|}{3}}|R(x-y)R(y)|dy+\int_{|y|\leq \frac{|x|}{3}}|R(x-y)R(y)|dy \\
&\leq \int_{|y|>\frac{|x|}{3}}|R(x-y)R(y)|dy+\int_{|x-y|> \frac{2|x|}{3}}|R(x-y)R(y)|dy \\
&\leq  M\left(\int_{|y|>\frac{|x|}{3}}|R(x-y)| g(y) dy+\int_{|x-y|> \frac{2|x|}{3}}|R(y)| g(x-y)dy \right)\\
&\leq  M g(x/3)\left(\int_{|y|>\frac{|x|}{3}}|R(x-y)|dy+\int_{|x-y|> \frac{2|x|}{3}}|R(y)|dy\right) \\
&\leq 2 M\|R\|_{L_1} g(x/3),
\end{aligned}
\end{equation*}
which implies \eqref{behavior of RR}. 

To prove \eqref{behavior of RR_2}, observe that
\begin{equation*}
\begin{aligned}
|(R\ast R)(x)| &\leq \int_{|y|<|x|/2}|R(x-y)R(y)|dy+\int_{|y|\geq|x|/2}|R(x-y)R(y)|dy \\
 &=\int_{|x-y|>|x|/2}|R(x-y)R(y)|dy+\int_{|y|\geq|x|/2}|R(x-y)R(y)|dy \\
 &\leq N\int_{|x-y|>|x|/2}|R(y)| h(x-y)dy+\int_{|y|\geq|x|/2}|R(x-y)| h(y)dy\\
 &\leq N h(x/2)\left(\int_{|x-y|>|x|/2}|R(y)|dy+\int_{|y|\geq|x|/2}|R(x-y)|dy\right)\\
 &\leq 2N\| R \|_{L_1} h(x/2),
\end{aligned}
\end{equation*}
which implies \eqref{behavior of RR_2}.

For the last assertion, note that  $R\in L_2(\bR^d)$ if $0<\alpha<d/2$, and apply H\"older's inequality. The lemma is proved.
\end{proof}\

Employing Lemma \ref{decay} yields 
Lemma \ref{4_key} that describes relations between diffusion coefficients and solutions to equation \eqref{targeteq:lipschitz:1}.
Recall that for $f\in\cS'_{n}$, there exists the corresponding measure $\mu\in M_T$ such that
\begin{equation*}
(f,\phi) = \int_{\bR^d} \phi(x) \mu(dx)\quad \mbox{for any}\quad \phi\in L_1(\mu), 
\end{equation*}
and
\begin{equation*}
\int_{\bR^d} \frac{1}{(1+|x|^2)^{k/2}} \mu(dx)<\infty
\end{equation*}
for some $k\in[0,\infty)$; see Theorem \ref{191024_03}.

\begin{lem}
\label{4_key}
(a) Let $f\in\cS'_{n}\cap\cS'_{nd}$. Assume that $d$, $\gamma$, and $s$ satisfy one of the following conditions;
\begin{enumerate}[(i)]
\item 
$d = 1$, $\gamma\in [0,1/2)$, and $s\in(\frac{1}{1-2\gamma},\infty]$;

\item 
$d\geq2$, $\gamma\in[0,1),$ $s\in(\frac{d}{1-\gamma},\infty]$, and 
\begin{equation} \label{s}
\int_{|x|<1}|x|^{\frac{s(1-\gamma-d)}{s-1}}\mu(dx)<\infty\quad  \left((1-\gamma-d)\frac{\infty}{\infty}:= 1-\gamma-d \right),
\end{equation}
where $\mu$ is the corresponding measure of $f$.

\end{enumerate}
Take $r, p < \infty$ such that
\begin{equation*} 
\frac{1}{s}+\frac{1}{r} = 1\quad \left(\frac{1}{\infty}:=0\right) \quad\mbox{and}\quad  2\leq 2r\leq p <\infty.
\end{equation*}
If $\xi \in L_{2s}(\bR^d)$ and $u\in L_p(\bR^d)$, then
\begin{equation*}
\|\xi u \boldsymbol{v}\|_{H_p^{\gamma-1}(l_2)} \leq A^{1/2s}I^{1/2r}\|\xi\|_{L_{2s}}\|u\|_{L_p},
\end{equation*}
where 
$\boldsymbol{v}=(v_1,v_2.\cdots)$, $v_i(x) = (f\ast e_i)(x)$,
\begin{equation*}
A=\int_{\bR^d}\frac{1}{(1+|x|^2)^{k/2}}\mu(dx), \quad\mbox{and}\quad I=\int_{\bR^d}(R_{1-\gamma}^r*R_{1-\gamma}^r)(x)(1+|x|^2)^{k(r-1)/2}\mu(dx).
\end{equation*}
(b) 
Let $f\in \cS'_{n}\cap C_{nd}$ . Take $s\in(1,\infty]$ and $\gamma\in[0,1)$ such that
\begin{equation*}
\frac{d}{1-\gamma}<s.
\end{equation*}
If  $\xi \in L_{s}(\bR^d)$ and $u\in L_p(\bR^d)$, then
\begin{equation*}
\|\xi u \boldsymbol{v}\|_{H_p^{\gamma-1}(l_2)} \leq N\|\xi\|_{L_s}\|R_{\gamma}\|_{L_r}\|u\|_{L_p},
\end{equation*}
where $N=\sup_{x\in\bR^d}|f(x)|$.

\end{lem}
\begin{proof}
Consider $s<\infty$, since the case $s = \infty$ can be handled similarly.
By Fubini's theorem and Parseval's identity,
\begin{equation}
\label{l_2}
\begin{aligned}
|(1-\Delta)^{-\frac{(1-\gamma)}{2}}(\xi u \boldsymbol{v})(x)|_{l_2}^2&=\sum_{k=1}^{\infty}|(1-\Delta)^{-\frac{(1-\gamma)}{2}}(\xi u v_k)(x)|^2 \\
  &=\sum_{k=1}^{\infty}|R_{1-\gamma}*(\xi u (f*e_k))(x)|^2 \\
  &=\sum_{k=1}^{\infty}|(R_{1-\gamma}(x-\cdot)\xi u,f*e_k)|^2 \\
  &=\sum_{k=1}^{\infty}|(f,(R_{1-\gamma}(x-\cdot)\xi u)*\tilde{e}_k)|^2 \\
  &=\sum_{k=1}^{\infty}\langle R_{1-\gamma}(x-\cdot)\xi u,e_k\rangle_{\cH}^2=\|R_{1-\gamma}(x-\cdot)\xi u\|_{\cH}^2.
\end{aligned}
\end{equation}
Observe that the last term of \eqref{l_2} equals to
\begin{equation}
\label{conv}
\int_{\bR^d\times\bR^d}R_{1-\gamma}(x-(y-z))R_{1-\gamma}(x-z)u(y-z)\tilde{u}(z)\xi(y-z)\tilde{\xi}(z)dz\mu(dy).
\end{equation}

{\it Proof of $(a)$.}  
From \eqref{l_2} and \eqref{conv}, applying H\"older's inequality and change of variables lead to
\begin{equation} \label{eqn.08.12.02}
\begin{aligned}
&|(1-\Delta)^{-\frac{(1-\gamma)}{2}}(\xi u \boldsymbol{v})(x)|_{l_2}^2\\
&\quad\leq (J_r(x))^{1/r} \times \left( \int_{\bR^d \times \bR^d} |\xi(y-z)\xi(-z)|^s(1+|y|^2)^{-k/2} dz\mu(dy)\right)^{1/s} \\
&\quad\leq \|\xi\|_{L_{2s}}^2 A^{1/s}(J_r(x))^{1/r},
\end{aligned}
\end{equation}
where
\begin{equation*}
J_r(x)=\int_{\bR^d \times \bR^d} |R_{1-\gamma}(y-z)R_{1-\gamma}(z)|^r|u(y-z+x)u(x-z)|^r(1+|y|^2)^{k(r-1)/2} dz\mu(dy).
\end{equation*}
Hence,
\begin{equation*}
\|\xi u \boldsymbol{v}\|_{H_p^{\gamma-1}(l_2)}^p=\int_{\bR^d}|(1-\Delta)^{-\frac{(1-\gamma)}{2}}(\xi u \boldsymbol{v})(x)|_{l_2}^p dx \leq \|\xi\|_{L_{2s}}^p A^{p/2s}\int_{\bR^d}(J_r(x))^{p/2r}dx.
\end{equation*}

Since $p\geq 2r$, by Minkowski's inequality,
\begin{equation*}
\int_{\bR^d}(J_r(x))^{p/2r}dx \leq \|u\|_{L_p}^p I^{p/2r},
\end{equation*}
where 
\begin{equation*}
I=\int_{\bR^d}|(R_{1-\gamma}^r*R_{1-\gamma}^r)(x)|(1+|x|^2)^{k(r-1)/2}\mu(dx).
\end{equation*}

It remains to show that $I<\infty$.
Since $R^r_\gamma(x)$ exponentially decays as $|x|\to \infty$, $R^r_{1-\gamma}\ast R^r_{1-\gamma}$ decreases exponentially by Lemma \ref{decay}. 
Therefore, it is enough to show that 
\begin{equation}
\label{near_zero_mu_integrable_lipschitz}
\int_{|x|\leq 1}\left|R^r_{1-\gamma}\ast R^r_{1-\gamma}(x)\right|\mu(dx)<\infty.
\end{equation}
In the case of (i), observe that $R_{1-\gamma}\in L_{2r}(\bR)$. Indeed, $R_{ 1-\gamma}$ has exponential dacay as $|x|\to \infty$, and 
$$ \int_{|x|<1}|R_{1-\gamma}(x)|^{2r}dx \leq N(r,\gamma) \int_{|x|<1}|x|^{-2r\gamma}dx<\infty
$$
by the choice of constants.
Therefore, by H\"older's inequality, we have
\begin{equation}
\label{near_zero_mu_integrable_lipschitz_sup}
\sup_{x \in \bR}|R^r_{1-\gamma}\ast R^r_{1-\gamma}(x)| \leq \|R_{1-\gamma}\|_{L_{2r}(\bR)}^{ 2r},
\end{equation}
which implies that $R_{1-\gamma}^r\ast R_{1-\gamma}^r$ is $\mu$-integrable near zero.

In the case of (ii), by Lemma \ref{decay}, $R^r_{1-\gamma}\ast R^r_{1-\gamma}(x) $ decays exponentially as $|x|\to \infty$. In addition, $R^r_{1-\gamma}\ast R^r_{1-\gamma}(x)$ is bounded by $|x|^{r(1-\gamma-d)}=|x|^{\frac{s(1-\gamma-d)}{s-1}}$ near zero. By \eqref{s}, we have
\begin{equation*}
\begin{aligned}
&\int_{|x|<1}|(R^r_{1-\gamma}\ast R^r_{1-\gamma})(x)|(1+|x|^2)^{k(r-1)/2}\mu(dx) \\
\quad &\leq N(d,k,r,\gamma)\int_{|x|<1}|x|^{\frac{s(1-\gamma-d)}{s-1}}\mu(dx)<\infty,
\end{aligned}
\end{equation*}
which implies that  $R_{1-\gamma}^r\ast R_{1-\gamma}^r$ is $\mu$-integrable near zero.
\\

{\it Proof of $(b)$.}
By Remark \ref{191108_01},  $f$ is a  bounded function. 
By boundedness of $f$, applying H\"older's inequality and Young's inequality leads to 
\begin{equation}
\label{eqn.08.12.03}
\begin{aligned}
&|(1-\Delta)^{-\frac{(1-\gamma)}{2}}(\xi u \boldsymbol{v})(x)|_{l_2}^2  \\
  &\quad \leq N\left( \int_{\bR^d \times \bR^d} |\xi(y-z)\xi(-z)|^s dzdy\right)^{1/s} \times (J_r'(x))^{1/r} \\
  & \quad \leq N \|\xi\|_{L_s}^2\times (J_r'(x))^{1/r},
\end{aligned}
\end{equation}
where 
\begin{equation*}
J_r'(x)=\int_{\bR^d \times \bR^d} |R_{1-\gamma}(y-z)R_{1-\gamma}(z)|^r|u(y-z+x)u(x-z)|^r dzdy
\end{equation*}
and $N=\sup_{x\in\bR^d}|f(x)|$. Therefore, by \eqref{eqn.08.12.03}
\begin{equation*}
\|\xi u \boldsymbol{v}\|_{H_p^{\gamma-1}(l_2)}^p=\int_{\bR^d}|(1-\Delta)^{-\frac{(1-\gamma)}{2}}(\xi u \boldsymbol{v})(x)|_{l_2}^p dx \leq N^{p/2}\|\xi\|_{L_{s}}^p \int_{\bR^d}(J_r'(x))^{p/2r}dx.
\end{equation*}
Since $p\geq 2r$, by Minkowski's inequality and Young's inequality, we have
\begin{equation*}
\int_{\bR^d}(J_r'(x))^{p/2r}dx \leq \|u\|_{L_p}^p\|R_{1-\gamma}\|_{L_r}^p.
\end{equation*}
The lemma is proved.
\end{proof}\

By employing Theorem \ref{key} and Lemma \ref{4_key}, we can prove Theorem \ref{case=0}. 
\begin{proof}[\bf{Proof of Theorem \ref{case=0}}.]
To prove the first assertion, it suffices to show that the assumptions in Theorem \ref{key} hold. In other words, we only need to prove that for any $\ep>0$, there exists a constant $N_{\varepsilon}>0$ such that
\begin{equation}
\label{0_goal}
 \|(h(u)-h(w))\xi\boldsymbol{v}\|_{H_p^{\gamma-1}(l_2)}\leq \varepsilon\|u-w\|_{H_p^{\gamma}}+N_{\varepsilon}\|u-w\|_{H_p^{\gamma-1}}
\end{equation}
for all $\omega\in\Omega$, $t\leq\tau$ and $u,w\in H_p^{\gamma}$. 

By Lemma \ref{4_key}, we have 
\begin{equation}
\label{est.g}
\|(h(u)-h(w))\xi\boldsymbol{v}\|_{H_p^{\gamma-1}(l_2)}  \leq N \|u-w\|_{L_p}\quad \text{and}\quad \|\xi h(0) \boldsymbol{v}\|_{\bH_p^{\gamma-1}(\tau)}\leq N\|h(0)\|_{\bL_p(\tau)},
\end{equation}
where $N$ is independent of $u$, $w$ and $h(0)$.  
Then applying Theorem \ref{prop_of_bessel_space} (\ref{condition_of_constants_interpolation}) and Young's inequality leads to \eqref{0_goal}. The theorem is proved.
\end{proof}

\mysection{Proof for case with the super-linear diffusion coefficient}
\label{sec:superlinear}

This section provides a proof of Theorem \ref{main}, which describes unique solvability and $L_p$-regularity of equation  \eqref{targeteq}. For readers' convenience, we recall equation \eqref{targeteq}:
\begin{equation}
\label{targeteq:1}
du =(a^{ij}u_{x^ix^j}+b^iu_{x^i}+cu)dt+\sum_{k=1}^{\infty}\xi|u|^{1+\lambda}(f\ast e_k)dw^k_t,
\quad (t,x) \in (0,\infty) \times \bR^d;
\quad u(0,\cdot) = u_0(\cdot),
\end{equation}
where $\lambda \geq 0$.
Theorem \ref{max_principle}, Lemma \ref{Lem_5_1}, and Lemma \ref{Non_explosion} are used to prove Theorem \ref{main}. 
Based on investigation described in Section \ref{sec:lambda=0}, Lemma \ref{Lem_5_1} suggests a local solution to equation \eqref{targeteq:1}. Then a global solution is constructed from the local solution with the help of Lemma \ref{Non_explosion}, 
which proves the existence and regularity of a solution for \eqref{targeteq:1}. In the construction of the global solution, Theorem \ref{max_principle} is employed to show that $\|u(t,\cdot)\|_{L_1(\bR^d)}$ is a continuous local martingale. Detailed proof for Theorem \ref{main} including the uniqueness of a solution is described at the end of this section. 
\\

\begin{thm}[Maximum principle]
\label{max_principle}
Let $\tau\leq T$ be a bounded stopping time. Assume that all the conditions of Theorem \ref{case=0} are satisfied with $s = \infty$ and $h(u)=u$. If $u\in\cH_p^{\gamma}(\tau)$ is the solution described in Theorem \ref{case=0}, 
then $u(t,\cdot)$ is nonnegative for all $t\leq\tau$ almost surely.
\end{thm}
\begin{proof}
Although similar to the proof of \cite[Theorem 4.2]{han2020boundary}, 
the proof of Theorem \ref{max_principle} is given in \ref{appendix} for the sake of completeness.  
\end{proof}

Lemma \ref{Lem_5_1} is used to construct a sequence of functions that approximate a solution to equation \eqref{targeteq:1}.

\begin{assumption}[$\tau$]
\label{xi'}
The coefficient $\xi$ is  $\mathcal{P}\times\mathcal{B}(\bR^d)$-measurable. Also, there exists a constant $K > 0$ such that
\begin{equation}
\label{xi_1'}
\| \xi(t,\cdot) \|_{L_\infty}\leq K  \quad  \quad \forall \omega\in\Omega,\, 
\forall t\leq\tau. 
\end{equation}
\end{assumption}

\begin{lem}
\label{Lem_5_1}
Let $\tau\leq T$ be a bounded stopping time, $f\in\cS'_n\cap\cS'_{nd}$ and $m \in \bN$.
Suppose  Assumptions  \ref{ass coeff''} and \ref{xi'}($\tau$) hold. Assume that $d$, $\lambda$, $\gamma$, $p$ and $f$ satisfy one of the conditions \eqref{main_thm_first_condi}-\eqref{main_thm_third_condi} in Theorem \ref{main}. Assume $u_0$ is in $ U_p^\gamma$ and $u_0$ is nonnegative. Consider
\begin{equation}
\label{eq_m}
\begin{aligned}
du=(a^{ij}u_{x^ix^j}+b^iu_{x^i}+cu)dt+\sum_{k=1}^{\infty}\xi|(-m)\vee u\wedge m|^{1+\lambda}v_kdw^k_t,
\quad 0<t\leq \tau;
\quad u(0,\cdot)=u_0.
\end{aligned}
\end{equation}
Then, equation \eqref{eq_m} with initial data $u(0,\cdot) = u_0$ has a unique solution $u_m\in\cH_p^{\gamma}(\tau)$ and $u_m$ is  nonnegative. 
Moreover, if Assumption \ref{xi'}($\tau$) is replaced with Assumption \ref{xi'}($\infty$), then 
$$u_m\in\cH_p^{\gamma}(T')$$
for any $T'<\infty$.
\end{lem}

\begin{proof}
By the mean-value theorem, for $u,v\in \bR$, we have
\begin{equation}
|(-m)\vee u\wedge m|^{1+\lambda}-|(-m)\vee v\wedge m|^{1+\lambda} \leq (1+\lambda)(2m)^{\lambda}|u-v|.
\end{equation}
Thus, $\xi$ and $h(u)=(-m)\vee u\wedge m$ satisfy Assumptions \ref{assumption_for_xi}($\tau,\infty$) and \ref{assumption_for_h}($\tau,p$), respectively. In the virtue of Theorem \ref{case=0}, equation \eqref{eq_m} has a unique solution $u_m\in\cH_p^{\gamma}(\tau)$.

To see the nonnegativity of $u_m$, define
\begin{equation}
\label{xi_m}
\xi_m(t,x):=\frac{\xi(t,x)|(-m)\vee u_m(t,x)\wedge m|^{1+\lambda}}{u_m(t,x)}\quad \left(\frac{0}{0}:=0\right).
\end{equation}
Note that $\xi_m$ is bounded, i.e., $\sup\limits_{\omega\in\Omega}\sup\limits_{t\leq\tau}\|\xi_m(t,\cdot)\|_{\infty}<\infty$.  Observe that $u_m$ satisfies

\begin{equation}
du_m=\left(a^{ij}(u_m)_{x^ix^j} + b^i (u_m)_{x^i} + cu_m\right) \,dt+\sum_{k=1}^{\infty}\xi_m u_mv_kdw^k_t, \quad 0<t\leq\tau; \quad u_m(0,\cdot)=u_0.
\end{equation}
 Therefore, by Theorem \ref{max_principle} we conclude that $u_m$ is  nonnegative. To prove the last assertion, notice that there exists a unique solution $\hat{u}_m\in \cH_p^{\gamma}(T')$ for $T'\geq\tau$, by the first assertion. From the uniqueness of a solution in $\cH_p^{\gamma}(\tau)$, we have $\hat{u}_m = u_m$ in $\cH_p^{\gamma}(\tau)$. The lemma is proved.
\end{proof}

\begin{rem}
 Suppose Assumption \ref{xi'}($\infty$) holds. Let $u_m$ be the solution of equation \eqref{eq_m}. Then, for fixed $\phi\in C_c^{\infty}$,
\begin{equation*}
M_t:=\sum_{k=1}^{\infty}\int_0^t(\xi|u_m\wedge m|^{1+\lambda}v_k,\phi)dw^k_t
\end{equation*}
is a square integrable martingale. Indeed,
\begin{equation*}
\begin{aligned}
&\sum_{k=1}^{\infty}\bE\int_0^t(\xi|u_m\wedge m|^{1+\lambda}v_k,\phi)^2dt\leq N \bE\sum_{k=1}^{\infty}\int_0^t(v_k,\phi)^2dt  \\
  & \quad\quad= tN\sum_{k=1}^{\infty}(f*e_k,\phi)^2= tN\|\phi\|^2_{\cH}<\infty,
\end{aligned}
\end{equation*}
where $N=N(K',m,\lambda)$.
\end{rem}

Note that the solution $u_m$ from Lemma \ref{Lem_5_1} satisfies equation \eqref{targeteq:1} in $\cH_p^\gamma(\tau_m)$, where
$$ \tau_m:=\inf\{ t\geq0: \sup_{x}|u_m(t,x)|\geq m \}.
$$

In the proof of Theorem \ref{main},  Lemma \ref{Non_explosion} is used to show the divergence of stopping times. Lemma \ref{Aux. function} is applied to prove Lemma \ref{Non_explosion}.

\begin{lem}
\label{Aux. function}
For $k = 1,2,\cdots$, denote
\begin{equation*}
\psi_k(x):=\frac{1}{\cosh(|x|/k)}.
\end{equation*}
Suppose Assumption \ref{ass coeff''} holds and let $K$ is the constant described in  Assumption  \ref{ass coeff''}. 
Then for all $\omega, x,t$, 
\begin{equation}
\label{negative}
\begin{aligned}
&(a^{ij}\psi_{k})_{x^ix^j }-(b^i\psi_{k})_{x^i}+(c-4K)\psi_{k}
\\&\quad
=a^{ij}\psi_{k x^ix^j}+(2a^{ij}_{x^j}-b^i)\psi_{kx^i}+(a^{ij}_{x^ix^j}-b^i_{x_i}+c-4K)\psi_k \leq 0.
\end{aligned}
\end{equation}
\end{lem}
\begin{proof}
Observe that
\begin{equation*}
\psi_{kx^i}(x)=-\psi_k(x)\frac{ x^i\tanh(|x|/k)}{k|x|}
\end{equation*}
and
\begin{equation*}
\begin{aligned}
\psi_{kx^ix^j}(x)&=\psi_k(x)\frac{x^ix^j}{k^2|x|^2}\tanh^2(|x|/k) -\psi_k(x)\frac{x^ix^j}{k^2|x|^2\cosh^2(|x|/k)}\\
&\quad +\psi_k(x)\left(\frac{x^ix^j}{|x|^3}-\frac{\delta_{ij}}{|x|}\right)\frac{\tanh(|x|/k)}{k},
\end{aligned}
\end{equation*}
where $\delta_{ij}$ is the Kronecker delta.  Condition \eqref{ellipticity''} and the fact that $\tanh(|x|)\leq |x|\wedge1$ lead to
\begin{equation*}
\begin{aligned}
a^{ij}\psi_{kx^ix^j}(x)&=\psi_k(x)\frac{a^{ij}x^ix^j}{k^2|x|^2}\tanh^2(|x|/k) -\psi_k(x)\frac{a^{ij}x^ix^j}{k^2|x|^2\cosh^2(|x|/k)}
  \\
  &\quad+\psi_k(x)\left(\frac{a^{ij}x^ix^j}{|x|^3}-\frac{a^{ii}}{|x|}\right)\frac{\tanh(|x|/k)}{k} \\
  &\leq\psi_k(x) \frac{K}{k^2 }\tanh(|x|/k)-\psi_k(x)\frac{ \kappa_0}{k^2\cosh^2(|x|/k)}
  +\psi_k(x)\frac{(K-d\kappa_0)\tanh(|x|/k)}{k|x|} 
  \\&\leq \frac{K\psi_k(x) \tanh(|x|/k)}{k^2 }+\frac{K\psi_k(x)}{k^2 } 
  \\&
  \leq \frac{2K\psi_k(x)}{k^2 }.
\end{aligned}
\end{equation*}
and
\begin{equation*}
\begin{aligned}
(2a^{ij}_{x^j}-b^i)\psi_{kx^i}(x)&=-(2a^{ij}_{x^j}-b^i)\psi_k(x)\frac{x^i\tanh(|x|/k)}{k|x|} \\
  &\leq \frac{(2|a^{ij}_{x^j}|+|b^i|)}{k}\psi_k(x).
\end{aligned}
\end{equation*}
Thus by assumption \eqref{bounded''},
\begin{equation*}
\begin{aligned}
&a^{ij}\psi_{kx^ix^j}(x)+(2a^{ij}_{x^j}-b^i)\psi_{kx^i}(x)+(a^{ij}_{x^ix^j}-b^i_{x_i}+c-4K)\psi_k(x) \\
& \quad\leq
\psi_k(x)\left[\frac{2K}{k^2 }+\frac{2|a^{ij}_{x^j}|+|b^i|}{k}+|a^{ij}_{x^ix^j}|+|b^i_{x^i}|+|c|-4K \right]\leq 0.
\end{aligned}
\end{equation*}
The lemma is proved.
\end{proof}

\begin{lem}
\label{Non_explosion}
Suppose the conditions of Theorem \ref{main} are satisfied. Let $u_m$ be the solution described in Lemma \ref{Lem_5_1}.
Then for any $T<\infty$,
\begin{equation}
\label{zero}
\lim_{R \to \infty}\sup_{m}\bP\left(\sup_{t\leq T}\sup_{x\in\bR^d}|u_m(t,x)|\geq R\right)=0.
\end{equation}
\end{lem}

\begin{proof} 
Let $K$ be the constant described in assumption \eqref{bounded''}.
By It\^o's formula, $\hat u_m:=u_me^{-4Kt}$ satisfies
\begin{equation*}
\begin{gathered}
d\hat u_m = \left(a^{ij}\hat u_{mx^ix^j}+b^i\hat u_{mx^i}+c\hat u_m-4K\hat u_m\right)dt+\sum_{k=1}^{\infty}\xi e^{-4Kt}|\hat u_m\wedge m|^{1+\lambda}v_kdw^k_t, \quad 0<t\leq\tau
\end{gathered}
\end{equation*}
with initial data $u_0$. Let $\psi_k$ be from Lemma \ref{Aux. function}.
By multiplying $\psi_k$ (see Remark \ref{schwartz}), using integration by parts, taking expectation, and applying inequality \eqref{negative}, we have the following: for any stopping time $\tau\leq T$,
\begin{equation*}
\begin{aligned}
&\bE[(u_m(\tau) e^{-4KT},\psi_k)] \leq \bE[(u_m(\tau)e^{-4K\tau},\psi_k)] \\
&\quad =\bE\left[(u_0,\psi_k) \right]+\bE\left[\int_0^{\tau}(u_m(t),(a^{ij}\psi_k)_{x^ix^j }-(b^i\psi_k)_{x^i}+(c-4K)\psi_k)e^{-4Kt}dt \right] \\
&\quad \leq \bE\left[(u_0,\psi_k) \right]\leq \bE[\|u_0\|_{L_1}]=\|u_0\|_{L_1(\Omega:L_1)}.
\end{aligned}
\end{equation*}
Then for any $\gamma\in(0,1)$ (e.g. \cite[Theorem III.6.8]{krylov1995introduction}),
\begin{equation*}
\bE\sup_{t\leq T}\left(\int_{\bR^d}u_m(t,x) \psi_k(x)dx\right)^{\gamma} \leq e^{4\gamma KT}\frac{2-\gamma}{1-\gamma}\|u_0\|_{L_1(\Omega:L_1)}^{\gamma}.
\end{equation*}
Since $u_m$ is nonnegative by Lemma \ref{Lem_5_1}, using the monotone convergence theorem as $k\to\infty$,
\begin{equation} \label{L_1 est.}
\bE\sup_{t\leq T}\|u_m(t,\cdot)\|_{L_1}^{1/2} \leq 3e^{2KT}\|u_0\|_{L_1(\Omega:L_1)}^{1/2}=:N_{1}.
\end{equation}
Thus, Chebyshev's inequality and inequality \eqref{L_1 est.} yield
\begin{equation*}
\bP\left(\sup_{t\leq T}\|u_m(t,\cdot)\|_{L_1}\geq S \right)\leq\frac{N_{1}}{\sqrt{S}},
\end{equation*}
where
$N_{1}$ is independent of $m$ and $S$. 
Fix $m$, $S>0$ and define
\begin{equation} \label{stopping_time_tau_m_s}
\tau_m(S):=\inf\{t\geq0:\|u_m(t,\cdot)\|_{L_1}\geq S\}.
\end{equation}
Note that
$\tau_m(S)$ is a stopping time by inequality \eqref{L_1 est.}. Besides, for any $T<\infty$, 
$u_m$ satisfies
\begin{equation}
\label{eq with xi_m}
du_m=\left( a^{ij}u_{mx^ix^j}+b^iu_{mx^i}+cu_m \right) \,dt+\sum_{k=1}^{\infty}\xi_m u_mv_kdw^k_t, \quad 0<t\leq\tau_m(S)\wedge T
\end{equation}
with initial data $u_0$, where $\xi_m$ is definded by \eqref{xi_m}.
Due to the definition of $\xi_m$, we can apply Theorem \ref{case=0} with equation \eqref{eq with xi_m}. Indeed, if either condition 
$(i)$ or $(ii)$ in Theorem \ref{case=0} holds, 
$$p>\frac{d+2}{\gamma}>\frac{2}{1-2\lambda}=\frac{2s}{s-1},$$ 
for $s:=\frac{1}{2\lambda}$. Thus for $t\leq \tau_m(S)$,
\begin{equation*} 
\|\xi_m(t,\cdot)\|_{L_{2s}}^{2s}\leq \int_{\bR^d}|\xi(t,x)|^{2s}|u_m(t,x)|^{2\lambda s}dx \leq K^{2s}S.
\end{equation*}
Hence, in each case, one can apply Theorem \ref{case=0} with $h(u) = u$ since $\xi_m$ satisfy Assumption \ref{assumption_for_xi} $(\tau,2s)$.

On the other hand, if condition (iii) holds, 
$$p>\frac{d+2}{\gamma}>\frac{2}{1-\lambda}=\frac{2s}{s-1},
$$
where $s:=\frac{1}{\lambda}.$ Then for $t\leq\tau_m(S)$,
\begin{equation*}
\|\xi_m(t,\cdot)\|_{L_{s}}^{s}\leq \int_{\bR^d}|\xi(t,x)|^{s}|u_m(t,x)|^{\lambda s}dx \leq K^{s}S.
\end{equation*}
Thus, applying Theorem \ref{case=0} with $h(u)=u$ and $\xi_m$ gives
\begin{equation}
\label{eqn.08.19.01}
\|u_m\|_{\cH_p^{\gamma}(\tau_m(S)\wedge T)}^p \leq N_{ S} \|u_0\|^p_{U_p^{\gamma}},
\end{equation}
where $N_{S}$ is independent of $m$. 
The estimate \eqref{eqn.08.19.01} and Corollary \ref{s.t. holder} lead to
\begin{equation*}
\bE\sup_{t\leq\tau_m(S)\wedge T}\sup_{x\in\bR^d}|u_m(t,x)|^p \leq N_S\|u_0\|^p_{U_p^{\gamma}}.
\end{equation*}
Therefore, by Chebyshev's inequality,
\begin{equation*}
\begin{aligned}
&\bP\left(\sup_{t\leq T}\sup_{x\in\bR^d}|u_m(t,x)|\geq R\right) \\
&\quad \leq  \bP\left(\sup_{t\leq\tau_m(S)\wedge T}\sup_{x\in\bR^d}|u_m(t,x)|\geq R\right)+\bP(\tau_m(S)\leq T) \\
&\quad \leq  \frac{N_S}{R}\|u_0\|^p_{U_p^{1-\gamma}}+\bP\left( \sup_{t\leq T}\|u_m(t,\cdot)\|_{L_1}\geq S \right) \\
&\quad \leq \frac{N_S}{R} \|u_0\|^p_{U_p^{1-\gamma}}+\frac{N_{ 1}}{\sqrt{S}}.
\end{aligned}
\end{equation*}
Observe that $N_{1}$ is independent of $m,R$ ans $S$. 
Taking supremum with respect to $m$ and letting $R\to\infty$ and $S\to\infty$ in order yield the final result \eqref{zero}. 
The lemma is proved.
\end{proof}

\begin{rem} \label{rmk:range_of_lambda}
In Theorem \ref{main}, each range of $\lambda$ is a sufficient condition for the unique solvability.
To obtain motivation for the condition, assume that the covariance $f$ is in $\cS'_{n} \cap \cS'_{nd}$ and $u$ is a solution to \eqref{targeteq:1} with smooth initial data $u_0$. 
Then for $t\leq \tau_m(s)$ (see \eqref{stopping_time_tau_m_s}), $u$ satisfies (at least formally)
\begin{equation} \label{motivation_of_range_of_lambda}
dv=(a^{ij}v_{x^ix^j}+b^iv_{x^i}+cv)dt+\xi(f\ast e_k)\frac{|u|^{1+\lambda}}{u}|v|dw_t^k, 
\quad (t,x)\in(0,\infty)\times\bR^d,
\end{equation}
with $u(0,\cdot) = u_0$ ($\frac{0}{0}:= 0$). To handle equation \eqref{motivation_of_range_of_lambda}, we consider Theorem \ref{case=0}. Since $\sup_{t\leq T}\| u(t,\cdot) \|_{L_1(\bR^d)}\leq S$ and the coefficient $\xi(f\ast e_k)\frac{|u|^{1+\lambda}}{u}$ have to satisfy Assumption \ref{assumption_for_xi} $(\tau_m(S),2s)$, we have
\begin{equation} \label{motivation_of_range_of_lambda_2}
\lambda = \frac{1}{2s}.
\end{equation}
Thus, we consider such $\lambda$ satisfying \eqref{motivation_of_range_of_lambda_2}. In the case of $f\in \cS'_{n} \cap C_{nd}$, equation \eqref{motivation_of_range_of_lambda_2} will be replaced with $\lambda  = \frac{1}{s}$.
\end{rem}

By employing Lemmas \ref{Lem_5_1} and \ref{Non_explosion}, we prove Theorem \ref{main}.
\begin{proof}[{\bf{Proof of Theorem \ref{main}}}]\
\label{pf:thm:main}
\\
{\it Step 1. Uniqueness. }
Suppose $u,v\in \cH_{p,loc}^{\gamma}$ are solutions of equation \eqref{targeteq:1}. By the definition of $\cH_{p,loc}^{\gamma}$, there are bounded stopping times $\tau_m$, $m=1,2,\cdots$ 
such that
\begin{equation*}
\tau_m\uparrow\infty\quad\mbox{and}\quad u,v\in \cH_{p,loc}^{\gamma}(\tau_m).
\end{equation*}
For fixed $m$, denote $\bar\tau := \tau_m$. 
By Corollary \ref{s.t. holder}, there exists small  $\ep>0$ such that
\begin{equation}
\label{uni.1}
\bE|u|^p_{C^{\ep}([0,\bar\tau]\times \bR^d )}+\bE|v|^p_{C^{\ep}([0,\bar\tau]\times \bR^d )}<\infty.
\end{equation}
For $n > 0$, define
\begin{equation*}
\begin{gathered}
\tau_n':=\inf\{t\leq\bar\tau:\sup_{x\in\bR^d}|u(t,x)|> n\}, \\
\tau_n'':=\inf\{t\leq\bar\tau:\sup_{x\in\bR^d}|v(t,x)|> n\},
\end{gathered}
\end{equation*}
and $\tilde{\tau}_n:=\tau_n'\wedge\tau_n''$.  
Due to \eqref{uni.1}, $\tau_n'$ and $\tau_n''$ are stopping times, and thus $\tilde{\tau}_n$ is a stopping time.
Observe that $\tilde{\tau}_n\uparrow \bar\tau$ (a.s.). 
Now, define $\hat\xi$ by
\begin{equation*}
\hat{\xi}:=\frac{\xi(t,x)|u|^{1+\lambda}}{u} \quad \left( \frac{0}{0}:=0\right).
\end{equation*}
Notice that
$$\sup\limits_{\omega\in\Omega}\sup\limits_{t\leq\tilde{\tau}_n}\|\hat{\xi}\|_{\infty}<\infty.
$$
Besides, $u$ is a solution to equation
\begin{equation*}
du=\left( a^{ij}u_{x^ix^j}+b^i u_{x^i}+cu \right)dt+\sum_{k=1}^{\infty}\hat{\xi}uv_kdw^k_t, \quad 0<t\leq\tilde{\tau}_n;\quad u(0,\cdot)=u_0.
\end{equation*}
By Theorem \ref{max_principle}, $u$ is nonnegative for all $t\leq\tilde{\tau}_n$ (a.s.). 
Similarly, $v$ is also nonnegative for all $t\leq\tilde{\tau}_n$ (a.s.).
Therefore, 
$$|(-n)\vee u \wedge n|=u \quad \text{and} \quad |(-n)\vee v \wedge n|=v,$$ 
for all $t\leq\tilde{\tau}_n$ (a.s.).
By the uniqueness result in Lemma \ref{Lem_5_1}, we conclude that $u=v$ in $\cH_p^{\gamma}(\tilde{\tau}_n)$ for each $n$. The monotone convergence theorem yields $u=v$ in $\cH_p^{\gamma}(\bar\tau)$.
\\

{\it{Step 2. Existence.}}  
By Lemma \ref{Lem_5_1}, for any $T<\infty$, there exists a nonnegative $u_m\in\cH_{p}^{\gamma}(T)$ such that 
\begin{equation}
du_m=\left( a^{ij}u_{mx^ix^j}+b^iu_{mx^i}+cu_m \right)\,dt+\sum_{k=1}^{\infty}\xi|u_m\wedge m|^{1+\lambda}v_kdw^k_t, \quad0<t\leq T
\end{equation}
and $u_m(0,\cdot)=u_0$. For nonnegative integers $m,R$ such that $m\geq R$, define
\begin{equation}
\label{stopping_time}
\tau_m^{R}:=\inf\{t>0:\sup_{x\in\bR^d}|u_m(t,x)|\geq R\}.
\end{equation}
Since $\sup\limits_{x\in\bR^d}|u_m(t,x)|\leq R$ for $t\leq \tau_m^R$,  we have $u_m\wedge m=u_m\wedge m\wedge R  = u_m \wedge R$ for $t\leq\tau_m^R$.  Thus, both $u_m$ and $u_R$ satisfy
\begin{equation*}
d\hat u=\left( a^{ij}\hat u_{x^ix^j}+b^i\hat u_{x^i}+c\hat u \right)\,dt+\sum_{k=1}^{\infty}\xi|\hat u\wedge R|^{1+\lambda}v_kdw^k_t, \quad0<t\leq\tau_m^R\wedge T; \quad \hat u(0,\cdot)=u_0.
\end{equation*}
 On the other hand, since $R\leq m$, $u_R\wedge R = u_R\wedge R\wedge m = u_R\wedge m$ for $t\leq \tau_R^R$. Then $u_m$ and $u_R$ satisfy
\begin{equation*}
d\hat u=\left( a^{ij}\hat u_{x^ix^j}+b^i\hat u_{x^i}+c\hat u \right)\,dt+\sum_{k=1}^{\infty}\xi|\hat u\wedge m|^{1+\lambda}v_kdw^k_t, \quad0<t\leq\tau_R^R\wedge T; \quad \hat u(0,\cdot)=u_0.
\end{equation*}
 By the uniqueness result in Lemma \ref{Lem_5_1}, $u_m=u_R$ in $\cH_p^{\gamma}((\tau_m^R\vee\tau_R^R)\wedge T)$ for any positive integer $T$.
 Therefore, we can conclude that $\tau_R^R=\tau_m^R\leq\tau_m^m$ (a.s.). Indeed, for $t<\tau_m^R$, 
\begin{equation*}
 \sup_{s\leq t}\sup_{x\in\bR^d}|u_R(s,x)|=\sup_{s\leq t}\sup_{x\in\bR^d}|u_m(s,x)|\leq R,
\end{equation*}
which implies  $\tau_m^R\leq\tau_R^R$. Similarly, we can obtain  $\tau_m^R\geq\tau_R^R$.  Now, observe that by Lemma \ref{Non_explosion}, 
\begin{equation*}
\begin{aligned}
\limsup_{m\to\infty}\bP(\tau_m^m\leq T)&=\limsup_{m\to\infty}\bP\left(\sup_{t\leq T,x\in\bR^d}|u_m(t,x)|\geq m\right) \\
  &\leq \limsup_{m\to\infty}\sup_{n}\bP\left(\sup_{t\leq T,x\in\bR^d}|u_n(t,x)|\geq m\right)\to 0,
\end{aligned}
\end{equation*}
which implies $\tau_m^m\to\infty$ in probability. Since $\tau_m^m$ is increasing,  we conclude that $\tau_m^m\uparrow\infty$ (a.s.).

Lastly, define
$\tau_m:=\tau_m^m\wedge m$ and
\begin{equation*}
u(t,x):=u_m(t,x)\quad \text{for}~ t\in[0,\tau_m).
\end{equation*}
Observe that $|u(t)|\leq m$ for $t\leq\tau_m$ and $u$ satisfies \eqref{targeteq:1} for all $t<\infty$.
Since $u=u_m$ for $t\leq\tau_m$ and  $u_m\in \cH_p^{\gamma}(\tau_m)$, it follows that $u\in\cH_p^{\gamma}(\tau_m)$ for any $m$,
and thus 
$u\in\cH_{p,loc}^{\gamma}$. The theorem is proved.

\end{proof}

\begin{proof}[{\bf{Proof of Theorem \ref{p_consistent}}}]\ \\
By the assumption, $u\in\cH_{p,loc}^\gamma$ is a solution to equation \eqref{targeteq}. By Definition \ref{def_of_sol_1} and Corollary \ref{s.t. holder}, there exists $\tau_n$ such that
\begin{equation}
\label{eqn.08.14.02}
\sup_{t\leq\tau_n}\sup_{x\in\bR^d}|u(t,x)|\leq n \quad \mbox{(a.s.)}.
\end{equation}
Thus, we have

\begin{equation*}
\begin{aligned}
\int_0^{\tau_n}\int_{\bR^d}|u(t,x)|^qdxdt &= \int_0^{\tau_n}\int_{\bR^d}|u(t,x)|^p|u(t,x)|^{q-p}dxdt \\
& \leq \left(\sup_{t\leq\tau_n} \sup_{x\in\bR^d}|u(t,x)|\right)^{q-p}\int_0^{\tau_n}\|u(t,\cdot)\|^p_{L_p}dt \\
& \leq  n^{q-p}\int_0^{\tau_n}\|u(t,\cdot)\|^p_{H_p^{\gamma}}dt<\infty \quad \mbox{(a.s.)}.
\end{aligned}
\end{equation*}
Therefore, we can define a bounded stopping time
\begin{equation*}
\tau_{n,k}:=\tau_n\wedge\inf\left\{t>0:\int_0^{t}\|u(t,\cdot)\|^q_{L_q}dt >k\right\}
\end{equation*}
such that $\tau_{n,k}\uparrow \tau_n$ as $ k\to\infty$ and $u \in\bL_{q}(\tau_{n,k})$ for each $k\in\bN$.

Note that Lemma \ref{4_key} and \eqref{eqn.08.14.02} imply
\begin{equation}
\label{regularity_up}
\|\xi|u|^{1+\lambda}\boldsymbol{v}\|_{\bH^{-1+\varepsilon}_q(\tau_{n,k},l_2)}\leq N\|u\|_{\bL_{q(1+\lambda)}(\tau_{n,k})}\leq N n^{\lambda}\|u\|_{\bL_{q}(\tau_{n,k})}
\end{equation}
for all $\varepsilon\in[0,1)$. 
Since $u_0\in U_q^0$, $a^{ij}u_{x^ix^j}+b^iu_{x^i}+cu\in\bH_q^{-2}(\tau_{n,k},l_2)$, and \eqref{regularity_up} implies $\xi|u|^{1+\lambda}\boldsymbol{v}\in \bH_q^{-1}(\tau_{n,k},l_2)$, we have $u\in\cH_q^0(\tau_{n,k})$. On the other hand, since \eqref{regularity_up} also yields $\xi|u|^{1+\lambda}\boldsymbol{v}\in \bH_q^{\gamma-1}(\tau_{n,k},l_2)$, by Theorem \ref{case=0}, there exists a unique solution $\hat u\in\cH^{\gamma}_q(\tau_{n,k})$ to equation
\begin{equation*}
d\hat u=\left(a^{ij}\hat u_{x^ix^j} + b^i \hat u_{x^i} + c \hat u\right)\,dt+\sum_{k=1}^{\infty}\xi|u|^{1+\lambda}v_kdw^k_t, \quad0<t\leq\tau_{n,k}; \quad \hat u(0,\cdot)=u_0.
\end{equation*}
It should be remarked that the stochastic part of the above equation is $\xi|u|^{1+\lambda}v_k$, not $\xi|\hat u|^{1+\lambda}v_k$.  Since $\varphi := u-\hat u\in \cH_q^0(\tau_{n,k})$ satisfies the equation
\begin{equation*}
 d\varphi=\left( a^{ij}\varphi_{x^ix^j} + b^i \varphi_{x^i} + c \varphi \right)\,dt, \quad0<t\leq\tau_{n,k}; \quad \varphi(0,\cdot)=0,
\end{equation*}
by Theorem \ref{key}, $u=\hat u$ almost all $(\omega,t,x)$. Therefore, $u\in\cH^{\gamma}_q(\tau_{n,k})$. 
The theorem is proved.

\end{proof}

\appendix
\mysection{Proof of maximum principle (Theorem \ref{max_principle})} \label{appendix}

Without loss of generality, we may assume $\tau = T$ (for example, consider $\xi1_{t\leq \tau}$ in place of $\xi$). Observe that there exists a sequence of nonnegative functions $u_0^n\in U_p^1$ such that
$u_0^n(\omega,\cdot)\in C_c^{\infty}$ for each $\omega\in\Omega$
and $u_0^n \rightarrow u_0$ in $U_p^{\gamma}$.
For $m=1,2,\cdots$, set
\begin{equation*}
\begin{gathered}
\boldsymbol{v}=(v_1,v_2,\cdots), \quad \boldsymbol{v}_m=(v_1,v_2\cdots,v_m,0,0,\cdots), \\
g(u)=\xi u \boldsymbol{v}=(\xi u v_1,\xi u v_2,\cdots), \\
g_m(u)=\xi u \boldsymbol{v}_m=(\xi u v_1,\xi u v_2\cdots,\xi u v_m,0,0,\cdots).
\end{gathered}
\end{equation*}
Obviously,
\begin{equation}
\label{finite_sum}
\|g_m(u)-g_m(v)\|_{H_p^{\gamma-1}(l_2)}\leq\|g(u)-g(v)\|_{H_p^{\gamma-1}(l_2)}.
\end{equation}
From the proof of Theorem \ref{case=0}, one can check that $g_m$ satisfies Assumption \ref{ass g}$(T)$. Therefore, there exists a unique solution $u_m\in \cH_p^{\gamma}(T)$ to  equation
\begin{equation*}
du_m=\left( a^{ij}u_{mx^ix^j}+b^iu_{mx^i}+cu_m \right)\,dt+\sum_{k=1}^{\infty}g_m^k(u_m)dw^k_t, \quad 0<t\leq T; \quad u_m(0,\cdot)=u_0^m,
\end{equation*}

Now, we prove that $u_m\rightarrow u$ in $\cH_p^{\gamma}(T)$ and $u_m$ is nonnegative.

First, we show that $u_m\rightarrow u$ in $\cH_p^{\gamma}(T)$. Note that $\varphi_m:=u-u_m$ satisfies $\varphi_m(0,\cdot)=u_0-u_0^m$,
\begin{equation*}
\begin{aligned}
d\varphi_m=\left( a^{ij}\varphi_{mx^ix^j}+b^i\varphi_{mx^i}+c\varphi_m \right) \,dt+\sum_{k=1}^{\infty}(g^k(u)-g_m^k(u_m))dw^k_t, \quad 0<t\leq T,
\end{aligned}
\end{equation*}
and
\begin{equation*}
g(u)-g_m(u_m)=g_m(u)-g_m(u_m)+\xi u (\boldsymbol{v}-\boldsymbol{v}_m).
\end{equation*}
By Theorem \ref{key} with $\tau=t$, \eqref{0_goal} and \eqref{finite_sum}, 
for any $t\leq T$,
\begin{equation*}
\begin{aligned}
&\|u-u_m\|^p_{\cH_p^{\gamma}(t)}  \leq N \|u_0-u^m_0\|^p_{U_p^{\gamma}}+N\|g_m(u)-g_m(u_m)\|^p_{\bH_p^{\gamma-1}(t,l_2)} 
+ N\|\xi u (\boldsymbol{v}-\boldsymbol{v}_m)\|^p_{\bH_p^{\gamma-1}(t,l_2)} 
\\&
\leq N \|u_0-u^m_0\|^p_{U_p^{\gamma}} + \frac{1}{2}\|u-u_m\|^p_{\bH_p^{\gamma}(t)}+N\|u-u_m\|^p_{\bH_p^{\gamma-1}(t)} 
+ N\|\xi u (\boldsymbol{v}-\boldsymbol{v}_m)\|^p_{\bH_p^{\gamma-1}(t,l_2)}.
\end{aligned}
\end{equation*}
Therefore,
\begin{equation*}
\|u-u_m\|^p_{\cH_p^{\gamma}(t)} \leq N \|u_0-u^m_0\|^p_{U_p^{\gamma}}+N\|u-u_m\|^p_{\bH_p^{\gamma-1}(t)}+N\|\xi u (\boldsymbol{v}-\boldsymbol{v}_m)\|^p_{\bH_p^{\gamma-1}(T,l_2)},
\end{equation*}
for any $t\leq T$.
By Gronwall's inequality and Theorem \ref{embedding} \eqref{gronwall-type},
\begin{equation}
\label{max_est}
\|u-u_m\|^p_{\cH_p^{\gamma}(T)} \leq N \|u_0-u^m_0\|^p_{U_p^{\gamma}}+N\|\xi u (\boldsymbol{v}-\boldsymbol{v}_m)\|^p_{\bH_p^{\gamma-1}(T,l_2)},
\end{equation}
where $N$ is independent of $m$ and $u$. Since $\|u_0-u^m_0\|^p_{U_p^{\gamma}}\rightarrow 0$ as $m\to\infty$, it is enough to show that $\|\xi u (\boldsymbol{v}-\boldsymbol{v}_m)\|^p_{\bH_p^{\gamma-1}(T,l_2)}$ goes to zero as $m\to\infty$. Note that
\begin{equation}
\label{eqn.08.13.01}
|(1-\Delta)^{-\frac{(1-\gamma)}{2}}(\xi u (\boldsymbol{v}-\boldsymbol{v}_m))(t,x)|^2_{l_2}=\sum_{k=m+1}^{\infty}\langle R_{1-\gamma}(x-\cdot)(\xi u)(t),e_k\rangle_{\cH}^2.
\end{equation}
 By the proof of Lemma \ref{4_key}, we have
\begin{equation*}
\|\xi(t) u(t) (\boldsymbol{v}-\boldsymbol{v}_m)\|_{H_p^{\gamma-1}(l_2)} \leq N\|u(t)\|_{H_p^{\gamma}},
\end{equation*}
where $N$ is independent of $m$.

Therefore, by  the dominated convergence theorem, 
\begin{equation*}
\|\xi(t) u(t) (\boldsymbol{v}-\boldsymbol{v}_m)\|_{H_p^{\gamma-1}(l_2)} \to 0
\end{equation*}
as $m\to\infty$. Again, by  the dominated convergence theorem,
\begin{equation*}
\|\xi u (\boldsymbol{v}-\boldsymbol{v}_m)\|_{\bH_p^{\gamma-1}(T,l_2)}  \to 0\quad \mbox{as}\quad m\to\infty.
\end{equation*}


Next, we prove that $u_m$ is  nonnegative. Observe that $u_m\in\bH_p^{\gamma}(T)\subseteq \bL_p(T)$. Since  $\xi $ and $ v_k$ are bounded, we conclude that
\begin{equation*}
g_m(u_m):=\xi u_m \boldsymbol{v}_m=(\xi u_m v_1,\xi u_m v_2\cdots,\xi u_m v_m,0,0,\cdots) \in \bL_p(T,l_2).
\end{equation*}
Since $u^m_0\in U_p^1 $, by Theorem \ref{key}, there exists a unique solution $\hat{u}_m\in\cH_p^1(T)$ to
\begin{equation}
\label{eqn.08.19.02}
d\hat{u}_m=\left(a^{ij}\hat u_{mx^ix^j}+b^i\hat u_{mx^i}+c\hat u_m \right) \,dt+\sum_{k=1}^{m}g_m^k(u_m)dw^k_t, \quad 0<t\leq T; \quad \hat{u}_m(0,\cdot)=u_0^m.
\end{equation}
It should be noted that the stochastic part of \eqref{eqn.08.19.02} is $g_m^k(u_m)$, not $g_m^k(\hat u_m)$.
Observe that $u_m$ and $\hat{u}_m$ are solutions to equation \eqref{eqn.08.19.02} in $\cH_p^{\gamma}(T)$.  By the uniqueness  in $\cH_p^{\gamma}(T)$, we have $u_m = \hat{u}_m$ in $\cH_p^{\gamma}(T)$. Thus, $u_m$ is in $\cH_p^{1}(T)$ and $u_m$ satisfies
\begin{equation*}
du_m = \left( a^{ij}u_{mx^ix^j}+b^iu_{mx^i}+cu_m \right)dt + \sum_{k\leq m}\xi  u_m v_k dw_t^k,\quad 0< t\leq T;\quad u_m(0,\cdot) = u_0^m.
\end{equation*}
Since $\xi v_k$ is bounded for each $k=1,2,\cdots,$ by maximum principle (e.g. \cite[Theorem 1.1]{krylov2007maximum}), we conclude that $u_m(t)$ is nonnegative for all $t\leq T$ almost surely. The theorem is proved.
\qed

\begin{acknowledgment}
The authors are sincerely grateful to the anonymous referees and Professor Kyeong-Hun Kim for giving many useful comments and suggestions. The authors also would like to thank Hee-Sun Choi for giving many helpful comments.
\end{acknowledgment}


\begin{thebibliography}{10}

\bibitem{arnold1981mathematical}
Arnold, L.,
\newblock {\em Mathematical models of chemical reactions},
\newblock {Stochastic Systems: The Mathematics of Filtering and Identification and Applications}, Springer, 1981, pages 111--134.

\bibitem{burdzy2010non}
Burdzy, K., Mueller, C., Perkins, E. A., 
\newblock{\em Nonuniqueness for nonnegative solutions of parabolic stochastic partial differential equations},
\newblock {Illinois Journal of Mathematics}, volume 54, number 4,  (2010), pages 1481--1507.



\bibitem{dalang1998stochastic}
Dalang, R. C.  and Frangos, N. E.,
\newblock {\em The stochastic wave equation in two spatial dimensions},
\newblock {Annals of Probability}, volume 26, number 1, (1998), pages 187--212.

\bibitem{dalang1999extending}
Dalang, R. C.,
\newblock {\em Extending the martingale measure stochastic integral with
  applications to spatially homogeneous spde's},
\newblock {Electronic Journal of Probability}, volume 4, (1999).

\bibitem{dawson1972stochastic}
Dawson, D.A.,
\newblock {\em Stochastic evolution equations},
\newblock {Mathmatical Biosciences}, volume 15, issues 3-4, (1972), pages 287--316.

\bibitem{dawson1980spatially}
Dawson, D. A. and Salehi, H.,
\newblock {\em Spatially homogeneous random evolutions},
\newblock {Journal of Multivariate Analysis}, volume 10, issue 2, (1980), pages 141--180.

\bibitem{ferrante2006spdes}
Ferrante, M.  and Sanz-Sol{\'e}, M.,
\newblock {\em SPDEs with coloured noise: analytic and stochastic approaches},
\newblock {ESAIM: Probability and Statistics}, volume 10, (2006), pages 380--405.




\bibitem{gel2014generalized}
Gel'fand, I. M.  and Vilenkin, N. Y.,
\newblock {\em Generalized functions volume 4 : Applications of harmonic analysis},
\newblock Academic press, 1964.

\bibitem{Kijung}
Gomez, A.,  Lee, K.,   Mueller, C.,   Wei, A. and   Xiong, J.,
\newblock {\em Strong uniqueness for an SPDE via backward doubly stochastic differential equations},
\newblock {Statistics \& Probability Letters}, volume 83, issue 10, (2013), pages 2186--2190.

\bibitem{grafakos2009modern}
Grafakos, L.,
\newblock {\em Modern fourier analysis}, volume 250.
\newblock Springer, 2009.

\bibitem{han2020boundary}
Han, B. and Kim, K.-H.,
\newblock {\em Boundary behavior and interior H{\"o}lder regularity of the solution
  to nonlinear stochastic partial differential equation driven by space-time
  white noise.}
\newblock {Journal of Differential Equations}, 269(11):9904--9935, 2020.

\bibitem{henry2006geometric}
Henry, D.,
\newblock{\em Geometric theory of semilinear parabolic equations\/}, volume 840,
\newblock Springer, 2006.

\bibitem{krein2002interpolation}
Krein, S. G.  and Semenov. E. M.,
\newblock {\em Interpolation of linear operators}, volume 54,
\newblock American Mathematical Soc., 2002.

\bibitem{krylov1995introduction}
Krylov, N. V.,
\newblock {\em Introduction to the theory of diffusion processes}, volume 142,
\newblock American Mathematical Soc. Tansl. Math. Monogr., Providence, RI, 1995.
  
\bibitem{kry99analytic}
Krylov, N. V.,
\newblock {\em An analytic approach to spdes},
\newblock {Stochastic partial differential equations: six perspectives}, volume 64, pages185--242, 1999.
  
\bibitem{krylov2007maximum}
Krylov, N. V.,
\newblock {\em Maximum principle for spdes and its applications},
\newblock {Stochastic Differential Equations: Theory And Applications: A
  Volume in Honor of Professor Boris L Rozovskii}, World Scientific, 2007, pages 311--338.
  
\bibitem{krylov2008lectures}
Krylov N. V.,
\newblock {\em Lectures on elliptic and parabolic equations in Sobolev spaces}, volume 96,
\newblock American Mathematical Soc., 2008.
  
\bibitem{ladyvzenskaja1988linear}
Lady{\v{z}}enskaja, O. A. , Solonnikov, V. A. and Ural'ceva, N. N.,
\newblock {\em Linear and quasi-linear equations of parabolic type}, volume 23,
\newblock American Mathematical Soc., 1988.

\bibitem{major2014multiple}
Major, P.,
\newblock {\em Multiple Wiener-It{\^o} integrals : With Applications to Limit Theorems, Second edition},
\newblock Springer, 2014.

\bibitem{mckean1970nagumo} 
McKean, H. P. Jr,
\newblock {\em Nagumo's equation}, 
\newblock {Advances in Mathematics}, volume 4, issue 3, (1970), pages 209--223.

\bibitem{mueller1991long}
Mueller, C.,
\newblock {\em Long time existence for the heat equation with a noise term},
\newblock {Probability theory and related fields}, volume 90, issue 4, (1991), pages 505--517.

\bibitem{mueller1999critical}
Mueller, C.,  
\newblock{\em The critical parameter for the heat equation with a noise term to blow up in finite time},
\newblock {Annals of Probability}, volume 28, number 4, (2000), pages 1735--1746.

\bibitem{mueller2009some}
Mueller, C.,
\newblock{\em Some tools and results for parabolic stochastic partial differential equations},
\newblock {A minicourse on stochastic partial differential equations}, Springer, 2009, pages 111--144.


\bibitem{mueller2014nonuniqueness}
Mueller, C., Mytnik, L., Perkins, E. A.,
\newblock{\em  Nonuniqueness for a parabolic SPDE with $\frac{3}{4}-\varepsilon $-H\"older diffusion coefficients},
\newblock {Annals of Probability}, volume 42, number 5, (2014), pages 2032--2112.

\bibitem{mytnik2011pathwise}
Mytnik, L.,  Perkins, E. A.,
\newblock{\em  Pathwise uniqueness for stochastic heat equations
with H\"older continuous coefficients: The white noise case}, 
\newblock {Probability Theory Related Fields},volume 149, issues 1-2, (2011), pages 1--96.

\bibitem{skorokhod1982studies}
Skorokhod, A. V. ,
\newblock {\em Studies in the theory of random processes}, volume 7021,
\newblock Courier Dover Publications, 1982.

\bibitem{walsh1986introduction}
Walsh, J. B.
\newblock {\em An introduction to stochastic partial differential equations},
\newblock {{\'E}cole d'{\'E}t{\'e} de Probabilit{\'e}s de Saint Flour
  XIV-1984}, Springer, 1986, pages 265--439.

\bibitem{Xiong}
Xiong, J.,
\newblock{\em Super-Brownian motion as the unique strong solution to an SPDE},
\newblock {Annals of Probability}, volume 41, number 2, (2013), pages 1030--1054.

\bibitem{yamada1971uniqueness}
Yamada, T., Watanabe, S.,
\newblock {\em On the uniqueness of solutions of stochastic differential equations},
\newblock Journal of Mathematics of Kyoto University, volume 11, number 1, (1971), pages 155--167.


\end{thebibliography}
\end{document}